\documentclass[leqno,12pt]{amsart} 
\usepackage{amssymb} 
\usepackage{braket}
\usepackage{bm}
\usepackage{amsmath}

\numberwithin{equation}{section}

\theoremstyle{plain} 
\newtheorem{theorem}{\noindent\bf Theorem}[section] 
\newtheorem{lemma}[theorem]{\noindent\bf Lemma}

\newtheorem{proposition}[theorem]{\noindent\bf Proposition}

\theoremstyle{definition}
\newtheorem{definition}[theorem]{\noindent\bf Definition}
\newtheorem{remark}[theorem]{\noindent\bf Remark}
\newtheorem{example}[theorem]{\noindent\bf Example}

\newcommand{\deldel}{\sqrt{-1}\partial \overline{\partial}}

\newcommand{\e}{\varepsilon}

\newcommand{\pote}{\theta_{v}^{(\omega)}}
\newcommand{\Pote}{\theta_{V}^{(\omega)}}
\newcommand{\Poteo}{\theta_{V}^{(\omega_{\mathrm{ref}})}}
\newcommand{\Z}{Z_{\mathfrak{M}}}
\newcommand{\Eone}{\mathcal{E}^{1}(P^{*})}
\newcommand{\bPSH}{\mathrm{PSH}_{b}(P^{*})}
\newcommand{\refe}{\mathrm{ref}}
\newcommand{\JfctORG}{J_{\mathrm{red}}}
\newcommand{\Jfct}{J_{V, \mathrm{red}}^{\sigma}}
\newcommand{\Pvol}{|P^{*}|_{V}^{\sigma}}

\begin{document}

\title[Multiplier Hermitian-Einstein metrics on Fano manifolds of
KSM-type]
{Multiplier Hermitian-Einstein metrics on Fano manifolds of KSM-type}
\author[Y. Nakagawa and S. Nakamura]
{Yasuhiro Nakagawa and Satoshi Nakamura}  

\thanks{ 
The second author was partly supported by Grant-in-Aid for JSPS
Fellow-ships for Young Scientists No.17J02783,
and is partly supported by JSPS KAKENHI Grant Number JP 21K20342.
}

\address{
Y.~Nakagawa:
Faculty of Advanced Science and Technology, Kumamoto University,
Kurokami 2-40-1, Chuo-ku, Kumamoto 860-8555,
Japan
}
\email{yasunaka@educ.kumamoto-u.ac.jp}

\address{
S.~Nakamura: (Former address)
National Institute of Technology, Numazu College, 
3600 Ooka, Numazu-shi, Shizuoka 410-8501, Japan
}
\email{satonakamura@numazu-ct.ac.jp}

\address{
S.~Nakamura: (Current address)
Department of Mathematics, Tokyo Institute of Technology,
2-12-1, Ookayama, Meguro-ku, Tokyo, 152-8551, Japan
}
\email{s.nakamura@math.titech.ac.jp}


\maketitle

\begin{abstract}
In this article we focus on multiplier Hermitian-Einstein metrics
introduced by Mabuchi
which include K\"{a}hler-Einstein metrics, K\"{a}hler-Ricci solitons and
Mabuchi solitons as special cases.
We also focus on KSM-manifolds, which are introduced by the first author
as toric bundles, to establish a criterion for the existence 
of multiplier Hermitian-Einstein metrics in terms of KSM-data.
An explicit example for a KSM-manifold admitting
a family of multiplier Hermitian-Einstein metrics
is constructed by using a continuous path 
connecting a K\"{a}hler-Ricci soliton with a Mabuchi soliton.
\end{abstract}


\tableofcontents

\section{Introduction}\label{Intro}

Finding a canonical metric on a manifold is a central problem in
differential geometry.
In particular the existence problem for K\"{a}hler-Einstein metrics is
one of the central topics.
Some generalizations of K\"{a}hler-Einstein metrics for Fano manifolds
with non-vanishing Futaki invariant
such as K\"{a}hler-Ricci solitons and Mabuchi solitons are discussed by
many experts.
Mabuchi~\cite{Ma03, Ma02a, Ma02b} introduced the notion of multiplier
Hermitian-Einstein metrics which include K\"{a}hler-Ricci solitons and
Mabuchi solitons as special cases.
In this paper we shall focus on multiplier Hermitian-Einstein metrics
(see also \cite{HaLi20} for generalized K\"{a}hler-Ricci solitons
studied very recently by Han and Li).
Let $M$ be an $n$-dimensional Fano manifold.
We fix a holomorphic vector field $V$ on $M$ and a
$V_{\mathrm{Im}}$-invariant K\"{a}hler metric $\omega_{0}\in 2\pi
c_{1}(M)$,
where $V_{\mathrm{Im}}=\frac{1}{2\sqrt{-1}}(V-\overline{V})$ is the
imaginary part of $V$. Let 
$$
\mathcal{K}=\set{\omega_{\varphi}:=\omega_{0}+\deldel\varphi\,|\,
\varphi\in{C^{\infty}(M)_{\mathbb{R}}} \quad\text{and}\quad
\omega_{\varphi}>0}
$$
be the set of all K\"{a}hler metrics in $2\pi c_{1}(M)$, and put
$\mathcal{K}_{V}=\set{\omega\in\mathcal{K}\,|\,
L_{V_{\mathrm{Im}}}\omega=0}$.
Here $C^{\infty}(M)_{\mathbb{R}}$ denotes the set of real valued smooth
functions on $M$.
For each $\omega\in\mathcal{K}_{V}$,
there exists a unique real valued function
$\theta_{V}^{(\omega)}\in C^{\infty}(M)_{\mathbb{R}}$ on $M$ such that
\begin{equation}\label{theta-normalization}
i_{V}\omega=\sqrt{-1}\overline{\partial}\theta_{V}^{(\omega)}
\quad\text{and}\quad \int_{X}\theta_{V}^{(\omega)}\omega^{n}=0,
\end{equation}
where $i_{V}\omega$ is the interior product of $V$ and $\omega$.
According to \cite{FM}, two real numbers
$\min_{M}\theta_{V}^{(\omega)}$ and $\max_{M}\theta_{V}^{(\omega)}$ are
independent of the choice of $\omega\in\mathcal{K}_{V}$.
Let $\sigma(s)$ be a real-valued smooth function on the interval
$I=(\alpha,\beta)\,
(-\infty\leqq\alpha<\min_{M}\theta_{V}^{(\omega)}
\leqq\max_{M}\theta_{V}^{(\omega)}<\beta\leqq+\infty)$
satisfying one of the following conditions:
(i) $\dot{\sigma}\leqq{0}\leqq\ddot{\sigma}$, 
(ii) $\ddot{\sigma} > 0$,
where $\dot{\sigma}$ and $\ddot{\sigma}$ are the first derivative and
the second derivative respectively.
We consider the Hermitian form
$$
\widetilde{\omega}:=
\omega\exp\left(-\frac{1}{n}\sigma(\theta_{V}^{(\omega)})\right).
$$
Without loss of generality, we can also assume
$\int_{M}\exp
(-\sigma(\theta_{V}^{(\omega)}))\omega^{n}=\int_{M}\omega^{n}$.
Mabuchi~\cite{Ma03, Ma02a, Ma02b} 
called a conformally K\"{a}hler metric
$\widetilde{\omega}$  a {\it multiplier Hermitian metric $($of type
$(\sigma,V)$$)$}.
Note that the multiplier Hermitian metric $\widetilde{\omega}$ can be
seen as an Hermitian metric on the holomorphic tangent bundle $TM$.
Then $\widetilde{\omega}$ defines the Hermitian connection
$$
\widetilde{\nabla}:=\nabla -\frac{\partial
(\sigma(\theta_{V}^{(\omega)}))}{n} \mathrm{id}_{TM},
$$
where $\nabla$ is the natural connection with respect to $\omega$.
The Ricci form $\mathrm{Ric}^{\sigma}_{V}(\omega)$ of
$(\widetilde{\omega}, \widetilde{\nabla})$ is equal to
$\mathrm{Ric}(\omega)+\deldel\sigma(\theta_{V}^{(\omega)})$,
where $\mathrm{Ric}(\omega)\in{2\pi}c_{1}(M)$ is the Ricci form for
$\omega$ defined by $-\deldel\log\omega^{n}$.
\begin{definition}
The conformally K\"{a}hler metric $\widetilde{\omega}$ is
a {\it  multiplier Hermitian-Einstein metric $($of type $(\sigma,V)$$)$}
if $\mathrm{Ric}^{\sigma}_{V}(\omega)=\omega$.
\end{definition}
Define the Ricci potential $\rho_{\omega}$ for $\omega$ as follows.
\begin{equation}\label{Ricci}
\mathrm{Ric}(\omega)-\omega=\deldel\rho_{\omega} \quad\text{and}\quad
\int_{X}(1-e^{\rho_{\omega}})\omega^{n}=0.
\end{equation}
In this terminology, $\widetilde{\omega}$ is multiplier
Hermitian-Einstein if and only if
$\rho_{\omega}+\sigma(\theta_{V}^{(\omega)})=0$.

Multiplier Hermitian-Einstein metrics of type $(\sigma,V)$ give
some well-known generalizations of K\"{a}hler-Einstein metrics. 
\begin{enumerate}
\item[(i)] When $\sigma$ is a constant function, 
a multiplier Hermitian-Einstein metric $\widetilde{\omega}$ gives an
K\"{a}hler-Einstein metric.
\item[(ii)] When $\sigma(s) =-s+C$, where $C$ is a constant,
the metric $\widetilde{\omega}$ gives {\it a K\"{a}hler-Ricci soliton}
in the sense that $\mathrm{Ric}(\omega)-\omega=L_{V}\omega$.
\item[(iii)] When $\sigma(s)=-\log(s+C)$, where $C$ is a constant
strictly greater than $\min_{M}\theta_{V}^{(\omega)}$, the metric
$\widetilde{\omega}$ gives a {\it Mabuchi soliton} (\cite{Hi19-2, Ma01, Ya21})
in the sense that $1-e^{\rho_{\omega}}$ is a potential function of a
holomorphic vector field with respect to $\omega$. 
In this case the positivity of
$\exp(-\frac{1}{n}\sigma(\theta_{V}^{(\omega)}))$ plays an important
role of the existence for Mabuchi solitons on toric Fano manifolds
(See \cite{Na19, Ya21}).
\end{enumerate}

As well as the theory of K\"{a}hler-Einstein metrics, multiplier
Hermitian-Einstein metrics do not always exist.
Let $\mathrm{Aut}(M)$ be the holomorphic automorphism group of $M$.
Futaki~\cite{Fu02} (see also \cite{QLi07}) introduced the character
$$
\mathrm{Fut}_{V}^{\sigma}(X):=\int_{M}X
\left(\rho_{\omega}+\sigma(\theta_{V}^{(\omega)})\right)
e^{-\sigma(\theta_{V}^{(\omega)})}\omega^{n}
$$
on the Lie algebra of the subgroup of $\mathrm{Aut}(M)$ consisting of
all elements $g$ such that $\mathrm{Ad}(g)V=V$,
and showed that the value $\mathrm{Fut}_{V}^{\sigma}(X)$ is independent
of the choice of $\omega$.
In this paper we call the invariant $\mathrm{Fut}_{V}^{\sigma}$ the
{\it $(\sigma,V)$-Futaki invariant}.
In particular, if $M$ admits a multiplier Hermitian-Einstein metric of
type $(\sigma,V)$
then the $(\sigma,V)$-Futaki invariant must vanish identically.
Thus the existence for multiplier Hermitian-Einstein metrics is
non-trivial.

Integrating the $(\sigma,V)$-Futaki invariant, we get a functional on
$\mathcal{K}_{V}$,
called the {\it $(\sigma,V)$-Ding functional} in this paper.
More explicitly, the $(\sigma,V)$-Ding functional is defined by
$$
\mathrm{Ding}_{V}^{\sigma}(\varphi):=-E_{V}^{\sigma}(\varphi)
-\log\left(\frac{1}{\int_{M}\omega_{0}^{n}}
\int_{M}e^{\rho_{\omega_{\varphi}}-\varphi}\omega_{0}^{n}
\right)
$$
for $\omega_{\varphi}=\omega_{0}+\deldel\varphi\in\mathcal{K}_{V}$
(see also \cite{QLi07}),
where $E_{V}^{\sigma}(\varphi)$ is the modified Monge-Amp\`{e}re energy
(\cite{BW14}) defined by the derivative
$$
\delta E_{V}^{\sigma}(\delta\varphi):=\frac{1}{\int_{M}\omega_{0}^{n}}
\int_{M}\delta\varphi
e^{-\sigma(\theta_{V}^{(\omega_{\varphi})})}\omega_{\varphi}^{n}
$$
and the normalization $E_{V}^{\sigma}(0)=0$.
It is easy to see that a critical point of $\mathrm{Ding}_{V}^{\sigma}$
defines a multiplier Hermitian-Einstein metric.
We shall see the $(\sigma,V)$-Ding functional plays an important role
for our results.

Now we state our motivation for our results.
In this paper, we also focus on {\it KSM-manifolds} 
which are Fano manifolds having the structure of certain toric bundle.
This notion was introduced by the first author \cite{Nakg19, Nakg15}.
See Section~\ref{KSM-mfd} for details about KSM-manifolds.
One of the typical examples for KSM-manifolds is a projective bundle
$\mathbb{P}(L\oplus\mathcal{O}_{W})$
where $L\to W$ is a line bundle over a Fano K\"{a}hler-Einstein manifold
$W$ and $\mathcal{O}_{W}$ is the trivial line bundle over $W$.
Such manifold was discussed to establish a characterization of the
existence for a non-homogeneous K\"{a}hler-Einstein metric or a
K\"{a}hler-Ricci soliton
in \cite{Koi90, KoSa86, KoSa88, Ma87, Sak86} in late 1980's. 
These results are stated in Section~\ref{MetricKSM}. 
The main purpose of this paper is to generalize
these results from view points of
multiplier Hermitian-Einstein metrics, general KSM-manifolds 
and both an algebraic and an analytic stability condition associated
with the $(\sigma,V)$-Ding functional.

We state our main theorems.
Let $\mathfrak{M}=(W;L_{1},L_{2},\dotsc,L_{l};P)$ be an
$(n,l)$-dimensional KSM-data and $Z_{\mathfrak{M}}$ be the associated
KSM-manifold.
By definition, $\Z$ has the structure of a fiber bundle over an
$n$-dimensional
Fano K\"{a}hler-Einstein manifold $W$
whose fiber is the $l$-dimensional toric Fano manifold associated with
an $l$-dimensional Fano polytope $P$.
Let $V$ be a holomorphic vector field on $Z_{\mathfrak{M}}$ as above.
In the following theorems, we assume $V$ is {\it fiber-directed}.
Namely, $V$ is assumed to be a holomorphic vector field induced by the
natural (fiber-directed) $\left(\mathbb{C}^{*}\right)^{l}$-action on
$Z_{\mathfrak{M}}$.

\begin{theorem}\label{ThmA}
{\rm (Theorem~\ref{ThmA2})}
Suppose a holomorphic vector field $V$ on a KSM-manifold $\Z$ is
fiber-directed.
Then $Z_{\mathfrak{M}}$
admits a multiplier Hermitian-Einstein metric
$$
\widetilde{\omega}=
\omega\exp\left(\frac{-1}{n+l}\sigma(\theta_{V}^{(\omega)})\right),
$$
if and only if for any $k=1,2,\dotsc, l$, we have
\begin{equation}\label{barycenter}
\int_{P^{*}}z_{k}\prod_{\alpha=1}^{n}(1+\langle\bm{\mu}_{\alpha},\bm{z}
\rangle)e^{-\sigma(\theta_{V}^{(\omega)})}\bm{dz}=0.
\end{equation}
Here the potential function $\theta_{V}^{(\omega)}$ is written as
$-\sum_{k=1}^{l}c_{k}z_{k}+C_{V}$ on the dual polytope $P^{*}$ of a Fano
polytope $P$,
where $z_{k}$ is the standard coordinate on
$P^{*}\subset\mathbb{R}^{l}$, and $c_{k}$ and $C_{V}$ are constants
uniquely determined by $(Z_{\mathfrak{M}},V)$.
\end{theorem}

We will further show that the integral condition in the above theorem is
equivalent to
the vanishing of the $(\sigma,V)$-Futaki invariant for fiber-directed
holomorphic vector fields.
Theorem~\ref{ThmA} generalizes the results in
\cite{Koi90, KoSa86, KoSa88, Ma87, Nakg19, Sak86}
stated in Section~\ref{MetricKSM}
from view points of general KSM-data and multiplier Hermitian-Einstein
metrics.
The existence for K\"{a}hler-Einstein metrics and K\"{a}hler-Ricci
solitons on homogeneous KSM-manifolds were studied in
\cite{De20, PS10}. 
More generally, multiplier Hermitian-Einstein metrics 
on homogeneous KSM-manifolds were studied in \cite{DH21}.
These works are based on the structure of a homogeneous toric bundle,
and depends on the representation theory of Lie groups.
On the other hand, since our argument depends on convex analysis and
combinatorics for toric geometry,
we can also deal with non-homogeneous KSM-manifolds.
See Section~\ref{Example} for an explicit example of a non-homogeneous
KSM-manifold.

\begin{remark}
It is not known whether we can relax the assumption that $V$ is
fiber-directed.
However, we believe that the extremal vector field for a KSM manifold
must be fiber-directed, for example.
\end{remark}

The following theorem is an algebraic stability version of
Theorem~\ref{ThmA}.

\begin{theorem}\label{ThmB}
Suppose a holomorphic vector field $V$ on a KSM-manifold $\Z$ is
fiber-directed.
Then $\Z$ is fiber-directed relative $(\sigma,V)$-D-polystable
if and only if the integral condition \eqref{barycenter} holds.
\end{theorem}

The fiber-directed relative $(\sigma,V)$-D-polystability is defined by
observing the asymptotic slope of the $(\sigma,V)$-Ding functional along
geodesics
associated with fiber-directed toric test configurations.
See Section~\ref{toricgeod}.
We emphasize that we do not use multiplier Hermitian-Einstein metrics to
prove Theorem~\ref{ThmB}.
Thus, our argument gives another interpretation for the known results
stated in Section \ref{MetricKSM}
from view point of the fiber-directed relative
$(\sigma,V)$-D-polystability.

The following theorem, which is not used multiplier Hermitian-Einstein
metrics to prove,
is an analytic stability version of Theorem~\ref{ThmA}.

\begin{theorem}\label{ThmC}
Suppose a holomorphic vector field $V$ on a KSM-manifold $\Z$ is
fiber-directed.
Then the $(\sigma,V)$-Ding functional
$D_{V}^{\sigma}$ on $\Z$
$($see \eqref{ding-KSM} for the definition of $D_{V}^{\sigma}$$)$
is coercive
if and only if the integral condition \eqref{barycenter} holds.
$($See Definition~\ref{def:coercive} for the definition of the
coercivity for $D_{V}^{\sigma}$.$)$
\end{theorem}

Han-Li \cite{HaLi20} established the equivalence among the existence
for a multiplier Hermitian-Einstein metric,
an algebraic stability condition for general Fano manifolds
and the coercivity for an energy functional.
In our cases,
by focusing on KSM-manifolds, we can prove Theorems~\ref{ThmA} and
\ref{ThmB}
by a method of  toric geometry which is independent of Han-Li's
argument.
We also mention that
Apostolov-Jubert-Lahdili \cite{AJL21} 
(see also \cite{DJ22} written by Delcroix-Jubert) discuss
multiplier Hermitian-Einstein metrics on semisimple principal toric
fibrations as an application of Han-Li's work.

The structure of this paper is as follows.
In Section~\ref{KSM-mfd}, we review about the notion of KSM-manifolds
and the known results on the existence for 
canonical K\"ahler metrics
on KSM-manifolds.
In Section~\ref{section:Futaki}, the $(\sigma, V)$-Futaki invariant
is discussed to obtain Theorem~\ref{ThmA}.
The algebraic and analytic stability is discussed in
Section~\ref{section:stability}
and \ref{section:coercive} to prove Theorems~\ref{ThmB} and \ref{ThmC}
respectively.
In Section~\ref{NUStable},
inspired by Yao's work \cite[Sections~6 and 7]{Ya21},
we discuss the non-uniformly stable case.
In this case we have
$\exp(-\frac{1}{n}\sigma(\theta_{V}^{(\omega)})) = 0$
at some point of a manifold.
Although there does not necessarily exist a multiplier
Hermitian-Einstein metric,
we construct a solution of an equation for the Monge-Amp\`{e}re measure
in the sense of Alexandrov
under an additional assumption for  $\sigma$ on non-uniform stable
KSM-manifolds.
Finally, in Section~\ref{section:example}, we construct an interesting
example of a KSM-manifold satisfying the integral condition
\eqref{barycenter}
by using a continuous path between a K\"{a}hler-Ricci soliton and a
Mabuchi soliton.
Moreover we see that the same idea of this construction yields a
non-uniformly stable KSM-manifold.

\subsection*{Acknowledgments}
The authors would like to thank Thibaut Delcroix for helpful comments.
They also thank the referee for numerous useful remarks and suggestions 
which improved the presentation of the paper.
\section{KSM-manifolds}\label{KSM-mfd}

The notion of KSM-manifolds was introduced by the first author
\cite{Nakg19, Nakg15}.
In this section we review definition, properties and examples for
KSM-manifolds.
We also review known results on the existence for canonical K\"{a}hler
metrics on KSM-manifolds.
The main reference in this section is the article \cite{Nakg19} written
by the first author.

We should firstly mention that the name of ``KSM'' comes from three
initials, Koiso, Sakane and Mabuchi.
They studied non-homogeneous K\"{a}hler-Einstein metrics in
\cite{Koi90, KoSa86, KoSa88, Sak86}
and toric K\"{a}hler-Einstein metrics in \cite{Ma87} in late 1980's.
Each study is one of the pioneering works in the theory of
K\"{a}hler-Einstein metrics on Fano manifolds.

\subsection{Toric Fano manifolds}

It is well-known that there is a one to one correspondence between toric
Fano manifolds and Fano polytopes
(see for instance \cite{CLSbook, Odbook}).
A convex polytope $P$ in $\mathbb{R}^{l}$ is called an $l$-dimensional
Fano polytope if it satisfies the following properties:
\begin{enumerate}
\item[(i)] $P$ is an integrable polytope, that is, the set
	   $\mathcal{V}(P)$ of vertices of $P$ is contained in
	   $\mathbb{Z}^{l}$.
\item[(ii)] The origin $\bm{0}$ is contained in the interior
	    $\mathrm{Int}(P)$ of $P$.
\item[(iii)] $P$ is a simplicial polytope, that is, each facet
	     (i.e. codimension one face) of $P$ is a simplex;.
\item[(iv)] For any facet of $P$, the set of its vertices
            $\{\bm{b}_{1},\dotsc,\bm{b}_{l}\}$ forms a
            $\mathbb{Z}$-basis of $\mathbb{Z}^{l}$.
\end{enumerate}
For an $l$-dimensional Fano polytope $P$, its dual polytope $P^{*}$ is
defined as
$$
P^{*}:=
\Set{\bm{z}\in\mathbb{R}^{l}\,|\,\langle\bm{z},\bm{y}\rangle\leqq{1}
\quad\text{for any}\quad\bm{y}\in P},
$$
where $\langle\cdot, \cdot\rangle$ is the standard inner product for
$\mathbb{R}^{l}$. Let
$P^{*}\cap\mathbb{Z}^{l}=\{\bm{a}_{0},\bm{a}_{1},\dotsc,\bm{a}_{N}\}$ be
the set of lattice points in $P^{*}$.
Then the toric Fano manifold associated with $P$, which is expressed as
$X_P$, is constructed by the following way.
For $\bm{t}=(t_{1},\dotsc,t_{l})\in (\mathbb{C}^{*})^{l}$ 
and $\bm{a}=(a_{1},\dots,a_{l})\in\mathbb{Z}^{l}$, 
set
$$
\bm{t}^{\bm{a}}=t_{1}^{a_{1}}\dotsm{t_{l}^{a_{l}}}\in\mathbb{C}^{*},
$$
and define an injective holomorphic map
$\varphi_{P}:(\mathbb{C}^{*})^{l}\to\mathbb{P}^{N}(\mathbb{C})$ by
$$
\varphi_{P}(\bm{t}):=
[\bm{t}^{-a_{0}}:\bm{t}^{-a_{1}}:\dotsc:\bm{t}^{-a_{N}}].
$$
Then we have $X_{P}=\overline{\varphi_{P}((\mathbb{C}^{*})^{l})}$ 
where it is the closure of
$\varphi_{P}((\mathbb{C}^{*})^{l})$ in $\mathbb{P}^{N}(\mathbb{C})$.

\subsection{KSM-manifolds}

Now we introduce the notion of KSM-manifolds which is a Fano manifold
and has the structure of a toric fiber bundle over a Fano
Einstein-K\"{a}hler manifold.

\begin{definition}\label{KSM-data}
The tuple $\mathfrak{M}=(W;L_{1},\dotsc,L_{l};P) $ is called an
$(n,l)$-dimensional {\it KSM-data}
if the following four conditions are satisfied:
\begin{enumerate}
\item[(i)] $W$ is an $n$-dimensional Fano manifold with an
	   K\"{a}hler-Einstein metric $\nu_{0}$, that is, there exists a
	   K\"{a}hler form $\nu_{0}$ on $W$ satisfying
	   $\mathrm{Ric}(\nu_{0})=\nu_{0}$.
\item[(ii)] $L_{i}$ is a holomorphic line bundle over $W$ for each
	    $i=1,\dotsc,l$, admitting a Hermitian metric $h_{i}$ whose
	    curvature form $-\deldel\log h_{i}$ has constant eigenvalues
	    $\mu_{1}^{(i)},\dotsc,\mu_{n}^{(i)}$ with respect to
	    $\nu_{0}$.
\item[(iii)] For any $w\in W$, there exists a neighborhood $U$ of $w$,
	     a holomorphic local coordinate $(z_{1},\dotsc,z_{n})$ of
	     $U$ and a holomorphic local frame $e_{i}$ for $L_{i}|_{U}$
	     $(i=1,\dotsc,l)$ satisfying
\begin{align*}
&\nu_{0}=\sqrt{-1}\sum_{\alpha=1}^{n}
dz_{\alpha}\wedge d\overline{z_{\alpha}},\\
&h_{i}(e_{i},e_{i})=1,\\
&d(h_{i}(e_{i},e_{i}))=0,\\
&-\deldel\log h_{i}=\sqrt{-1}\sum_{\alpha=1}^{n}\mu_{\alpha}^{(i)}
dz_{\alpha}\wedge d\overline{z_{\alpha}}
\end{align*}
at $w$, simultaneously.
Here, for all $\alpha=1,\dotsc,n$, we put
$$
\bm{\mu}_{\alpha}:=(\mu_{\alpha}^{(1)},\dotsc,
\mu_{\alpha}^{(l)})\in\mathbb{R}^{l}
$$
which are called the curvature vectors of $\mathfrak{M}$.
\item[(iv)] $P$ is an $l$-dimensional Fano polytope satisfying
	    $-\bm{\mu}_{\alpha}\in\mathrm{Int}(P)$ for all
	    $\alpha=1,\dotsc,n$.
\end{enumerate}
\end{definition}

Fix an $(n,l)$-dimensional KSM-data
$\mathfrak{M}=(W;L_{1},\dotsc,L_{l};P)$
to define a KSM-manifold.
Let $\pi_{\mathfrak{M}}:Q_{\mathfrak{M}}\to W$ be the
$(\mathbb{C}^{*})^{l}$-bundle over $W$
associated with $L_{1}\oplus\dotsm\oplus L_{l}$.
Namely any element $\bm{q}\in Q_{\mathfrak{M}}$ satisfies
$\bm{q}=q_{1}\oplus\dotsm\oplus q_{l}$ and
$q_{i}\in L_{i}\setminus(\text{zero-section})$.
For $\bm{a}=(a_{1},\dotsc,a_{l})\in\mathbb{Z}^{l}$ and
$\bm{q}=q_{1}\oplus\dotsm\oplus q_{l}\in Q_{\mathfrak{M}}$, 
we put 
$$
\bm{q}^{\bm{a}}:=q_{1}^{\otimes a_{1}}\otimes\dotsm\otimes
q_{l}^{\otimes a_{l}}\in \bm{L}^{\bm{a}}\setminus\text{(zero-section)},
$$
where
$\bm{L}^{\bm{a}}:=L_{1}^{\otimes a_{1}}\otimes\dotsm\otimes
L_{l}^{\otimes a_{l}}$,
and also put
$$
E_{\mathfrak{M}}:=
\bigoplus_{\bm{a}\in P^{*}\cap\mathbb{Z}^{l}} \bm{L}^{-\bm{a}}
=\bm{L}^{-\bm{a}_{0}}\oplus\bm{L}^{-\bm{a}_{1}}\oplus\dotsm\oplus
\bm{L}^{-\bm{a}_{N}}
$$
to define a map
$$
\Phi_{\mathfrak{M}}: Q_{\mathfrak{M}}\to 
\mathbb{P}(E_{\mathfrak{M}}):=
(E_{\mathfrak{M}}\setminus (\text{zero-section}))/\mathbb{C}^{*}
$$
by
$\Phi_{\mathfrak{M}}(\bm{q}):=[\bm{q}^{-\bm{a}_{0}}\oplus
\bm{q}^{-\bm{a}_{1}}\oplus\dotsm\oplus\bm{q}^{-\bm{a}_{N}}]$.
Note the map $\Phi_{\mathfrak{M}}$ is holomorphic and injective.
Then we define {\it the KSM-manifold} $Z_{\mathfrak{M}}$ associated with
$\mathfrak{M}$ by
$$
Z_{\mathfrak{M}}:=
\overline{\Phi_{\mathfrak{M}}(Q_{\mathfrak{M}})}\subset
\mathbb{P}(E_{\mathfrak{M}}).
$$
Note the KSM-manifold $Z_{\mathfrak{M}}$ is an $(n+l)$-dimensional
complex manifold with the structure of an $X_{P}$-bundle
$\widetilde{\pi}_{\mathfrak{M}}: Z_{\mathfrak{M}}\to W$.

Furthermore every KSM-manifold is Fano.

\begin{theorem}\label{KSMFano}
{\sc (\cite[Theorem~2.12]{Nakg19})}
For any $(n,l)$-dimensional KSM-data $\mathfrak{M}$, the KSM-manifold
$Z_{\mathfrak{M}}$ has the positive first Chen class.
\end{theorem}

\begin{proof}
We give a proof for reader's convenience.
For $q_{i}\in L_{i}\setminus\text{(zero-section)}$, put 
\begin{equation}\label{xq}
x_{i}(q_{i}):=-\log{h_{i}}(q_{i},q_{i}).
\end{equation}
The pair $\bm{x}=(x_{1},\dotsc, x_{l})$ can be seen as a function on
$Q_{\mathfrak{M}}$.
We define a function $u_{P}$ for
$\bm{y}=(y_{1},\dotsc,y_{l})\in\mathbb{R}^{l}$ by
\begin{equation}\label{uP}
u_{P}(\bm{y}):=
\log\left(\sum_{\bm{a}\in P^{*}\cap\mathbb{Z}^{l}}
e^{\langle \bm{a},\bm{y}\rangle}\right)
\end{equation}
and define a volume form $\eta_{\refe}$ on $Q_{\mathfrak{M}}$ by
$$
\eta_{\refe}:=\frac{(n+l)!}{n!}e^{-u_{P}(\bm{x})}
\bigwedge_{i=1}^{l}\left(\sqrt{-1}\frac{d\tau_{i}}{\tau_{i}}\wedge
\frac{d\overline{\tau_{i}}}{\overline{\tau_{i}}}\right)
\wedge(\pi^{*}_{\mathfrak{M}}\nu_{0})^{n},
$$
where $\tau_{i}$ is the fiber coordinate for $L_i|_U$ satisfying
$q_{i}=\tau_{i}e_{i}$.
Note the volume form $\eta_{\refe}$ naturally extends on
$Z_{\mathfrak{M}}$.
Now we consider the real $(1,1)$-form
\begin{equation}\label{ref}
\omega_{\refe}:=
-\deldel\log\eta_{\refe}\in{2\pi}c_{1}(Z_{\mathfrak{M}})
\end{equation}
on $Z_{\mathfrak{M}}$.
Use local coordinates and local frames introduced in
Definition~\ref{KSM-data} and the condition
$\mathrm{Ric}(\nu_{0})=\nu_{0}$ to obtain 
\begin{equation}\label{ref2}
\omega_{\refe}=
\sqrt{-1}\sum_{i,j=1}^{l}
\frac{\partial^{2}u_{P}}{\partial{y_{i}}\partial{y_{j}}}(\bm{x}) 
\frac{d\tau_{i}}{\tau_{i}}\wedge
\frac{d\overline{\tau_{j}}}{\overline{\tau_{j}}} 
+\sqrt{-1}\sum_{\alpha=1}^{n}
\left(1+\langle\bm{\mu}_{\alpha},\bm{m}_{0}(\bm{x})\rangle\right)
dz_{\alpha}\wedge d\overline{z_{\alpha}},
\end{equation}
where
$\bm{m}_{0}(\bm{x})=
\left(\frac{\partial{u_{P}}}{\partial{y_{1}}}(\bm{x}),
\dotsc,\frac{\partial{u_{P}}}{\partial{y_{l}}}(\bm{x})\right)$
is the moment map for the $(\mathbb{C}^{*})^{l}$-action on the fiber
$\widetilde{\pi}^{-1}_{\mathfrak{M}}(w)\subset{Z_{\mathfrak{M}}}$ for
$w\in W$ with respect to the K\"{a}hler form
$\deldel
\left(u_{P}(\bm{x})|_{\widetilde{\pi}^{-1}_{\mathfrak{M}}(w)}\right)$.
We note that $\bm{m}_{0}(\widetilde{\pi}^{-1}_{\mathfrak{M}}(w))=P^{*}$
(see for instance \cite{Ma87}).
It follows from definition of $u_P$ and the condition
$-\bm{\mu}_{\alpha}\in\mathrm{Int}(P)$ that $\omega_{\refe}$ is
positive. This completes the proof.
\end{proof}

\subsection{Examples}\label{Example}

We see typical examples of KSM-manifolds.

\begin{example}
If $n=0$, that is, $W$ is the one point space, then a KSM manifold is
nothing but a toric Fano manifold $X_{P}$.
\end{example}

\begin{example}
If $l=1$, then the Fano polytope is only the closed interval
$[-1,1]\subset\mathbb{R}$.
For an $(n,1)$-dimensional KSM-data $\mathfrak{M}=(W;L;[-1,1])$,
the associated KSM manifold $Z_{\mathfrak{M}}$ is
$\mathbb{P}(L\oplus\mathcal{O}_{W})$
which has the structure of a $\mathbb{P}^{1}(\mathbb{C})$-bundle over
$W$. Here $\mathcal{O}_{W}$ is the trivial line bundle over $W$.
In particular, the $(1,1)$-dimensional KSM data
$\mathfrak{M}=(\mathbb{P}^{1}(\mathbb{C});
\mathcal{O}_{\mathbb{P}^{1}}(1);[-1,1])$
gives the one point blow up
$Z_{\mathfrak{M}}=\mathbb{P}^{2}(\mathbb{C})\#
\overline{\mathbb{P}^{2}(\mathbb{C})}$ of the projective plane.
\end{example}

\begin{example}
Let $P_{2}$ be the convex hull of $\{(1,0),(0,1),(-1,-1)\}$.
This is a Fano polytope which gives $\mathbb{P}^{2}(\mathbb{C})$.
The $(1,2)$-dimensional KSM-data
$\mathfrak{M}=(\mathbb{P}^{1}(\mathbb{C});
\mathcal{O}_{\mathbb{P}^{1}}(-1),\mathcal{O}_{\mathbb{P}^{1}};P_{2})$
gives the projective bundle
$Z_{\mathfrak{M}}
=\mathbb{P}(\mathcal{O}_{\mathbb{P}^{1}}(-1)\oplus
\mathcal{O}_{\mathbb{P}^{1}}\oplus\mathcal{O}_{\mathbb{P}^{1}})$.
On the other hand, the data
$(\mathbb{P}^{1}(\mathbb{C});\mathcal{O}_{\mathbb{P}^{1}}(1),
\mathcal{O}_{\mathbb{P}^{1}}; P_{2})$ is not a KSM one,
since the condition (iv) in Definition~\ref{KSM-data} does not hold.
In fact, the associated manifold 
$\mathbb{P}(\mathcal{O}_{\mathbb{P}^{1}}(1)\oplus
\mathcal{O}_{\mathbb{P}^{1}}\oplus\mathcal{O}_{\mathbb{P}^{1}})$ is not
Fano (see for instance \cite{CLSbook, Odbook}).
Note that our definition of projective bundles is different from that in
\cite{CLSbook, Odbook}.
For example, we define
$$
\mathbb{P}(\mathcal{O}_{\mathbb{P}^{1}}(1)\oplus
\mathcal{O}_{\mathbb{P}^{1}}\oplus\mathcal{O}_{\mathbb{P}^{1}})
:=\left(
(\mathcal{O}_{\mathbb{P}^{1}}(1)\oplus\mathcal{O}_{\mathbb{P}^{1}}
\oplus\mathcal{O}_{\mathbb{P}^{1}})\setminus(\text{zero-section})
\right)/\mathbb{C}^{*}.
$$
\end{example}

A KSM-data $\mathfrak{M}=(W;L_{1},\dotsc,L_{l};P)$ is said to be
{\it homogeneous} if the following conditions are satisfied: 

\begin{enumerate}
\item[(i)] $W$ is a simply connected compact complex homogeneous
	   K\"{a}hler manifold (i.e. K\"{a}hler C-space).
	   Then we have $W=G/U=G_{c}/K$, where $G$ is a simply connected
	   complex semisimple Lie group, $U$ is a parabolic subgroup of
	   $G$,
	   $G_{c}$ is a maximal compact subgroup of $G$ and
	   $K=G_{c}\cap{U}$.
\item[(ii)] The K\"{a}hler-Einstein metric $\nu_{0}$ is
	    $G_{c}$-invariant.
\item[(iii)] Each line bundle $L_{i}$ admits a $G$-action compatible
	     with the $G$-action on $W$, and the Hermitian metric
	     $h_{i}$ is $G_{c}$-invariant.
\end{enumerate}
The existence problem for K\"{a}hler-Einstein metrics and
K\"{a}hler-Ricci solitons on homogeneous KSM-manifolds are discussed by
Podest\`{a}-Spiro~\cite{PS10} and Delcroix~\cite{De20} from view points
of homogeneous toric bundles.
More generally, Delcroix-Hultgren~\cite{DH21} discuss multiplier
Hermitian-Einstein metrics on
homogeneous KSM-manifolds.
These works depend on the representation theory of Lie groups.
On the other hand, our arguments are based on convex analysis and
combinatorics for toric geometry.
Thus we can deal with a non-homogeneous KSM-manifold as follows.

\begin{example}
Let $\mathcal{M}_{n}$ be the moduli space of smooth hypersurfaces of
degree $n$ in $\mathbb{P}^{n+1}(\mathbb{C})$
and let $\mathcal{M}_{n}^{\mathrm{EK}}\subset\mathcal{M}_n$ be the
moduli space of K\"{a}hler-Einstein hypersurfaces.
The space $\mathcal{M}_{n}^{\mathrm{EK}}$ is non-empty (for instance the
Fermat hypersurface
$\left\{\sum_{i=0}^{n+1}(X_{i})^{n}=0\right\}$ is in
$\mathcal{M}_{n}^{\mathrm{EK}}$) and
open in $\mathcal{M}_{n}$.
For fixed $W\in\mathcal{M}_{n}^{\mathrm{EK}}$, put
$L=\mathcal{O}_{\mathbb{P}^{n+1}}(1)|_{W}$.
Then $K_{W}^{-1}\cong L^{\otimes 2}$ by the adjunction formula.
Take an $l$-dimensional Fano polytope $P$ and integers
$k_{1},\dotsc,k_{l}\in\mathbb{Z}$
such that $-\frac{1}{2}(k_{1},\dotsc,k_{l})\in\mathrm{Int}(P)$.
Then $\mathfrak{M}=(W;L^{\otimes{k_{1}}},\dotsc,L^{\otimes{k_{l}}};P)$
is an $(n,l)$-dimensional KSM data,
and the KSM manifold $Z_{\mathfrak{M}}$ is non-homogeneous in general.
\end{example}

Other explicit examples of higher dimensional KSM-manifolds are also
discussed by the first author.
In \cite{Nakg15}, K\"{a}hler-Einstein KSM-manifolds which are not
asymptotically Chow semi-stable are constructed.

\subsection{Existence for canonical K\"{a}hler metrics on
KSM-manifolds}\label{MetricKSM}

Finally we review known theorems on the existence for canonical
K\"{a}hler metrics on KSM-manifolds which motivate our results.

\begin{theorem}
Let $\mathfrak{M}=(W;L;[-1,1])$ be an $(n,1)$-dimensional KSM-data and
 $Z_{\mathfrak{M}}$ be the associated KSM-manifold.
\begin{enumerate}
\item[(i)] {\rm (Sakane~\cite{Sak86},
	   Koiso-Sakane~\cite{KoSa86, KoSa88}, Mabuchi~\cite{Ma87})} 
	   The KSM-manifold $Z_{\mathfrak{M}}$ admits a
	   K\"{a}hler-Einstein metric if and only if
$$
\int_{-1}^{1}z\prod_{\alpha=1}^{n}(1+\mu_{\alpha}z)dz=0.
$$
\item[(ii)] {\rm (Koiso~\cite{Koi90})}
	    The KSM-manifold $Z_{\mathfrak{M}}$ always admits a
	    K\"{a}hler-Ricci soliton with fiber-directed soliton vector
	    field.
\end{enumerate}
\end{theorem}

The first author considered general KSM-data to obtain the following
theorem which also generalizes a Wang-Zhu's work \cite{WaZh04} on the
existence for K\"{a}hler-Ricci solitons on toric Fano manifolds.
\begin{theorem}
{\rm (\cite{Nakg19})}
The KSM-manifold $Z_{\mathfrak{M}}$ associated to an $(n,l)$-dimensional
KSM-data $\mathfrak{M}=(W;L_{1}, L_{2},\dotsc,L_{l};P)$ always admits
a K\"{a}hler-Ricci soliton $\omega$ with fiber-directed soliton vector
field.
Furthermore $\omega$ becomes a K\"{a}hler-Einstein metric if and only if
for any $k=1,2,\dotsc,l$, we have
$$
\int_{P^{*}}z_{k}\prod_{\alpha=1}^{n}(1+\langle \bm{\mu}_{\alpha},
\bm{z} \rangle ) \bm{dz}=0.
$$
\end{theorem}

Therefore our result in Theorem~\ref{ThmA} generalizes the above
theorems from view points of general KSM-data and multiplier
Hermitian-Einstein metrics.

\section{Multiplier Hermitian-Einstein metrics and the
$(\sigma,V)$-Futaki invariant on KSM-manifolds}\label{section:Futaki}

The goal in this section is to prove Theorem~\ref{ThmA}.
We first prove some lemmas for the $(\sigma,V)$-Futaki invariant for
general Fano manifolds to express this invariant in terms of KSM-data.

\subsection{The $(\sigma,V)$-Futaki invariant for Fano manifolds}

Let $M$ be an $n$-dimensional Fano manifold with a K\"{a}hler metric
$\omega\in 2 \pi c_{1}(M)$.
Recall the Ricci potential for $\omega$ is denoted by $\rho_{\omega}$.
Fix a holomorphic vector field $V$ on $M$ to consider a multiplier
Hermitian metric of type-$(\sigma,V)$.
As in the Section~\ref{Intro}, for any holomorphic vector field $v$ on
$M$, its potential function $\theta_{v}^{(\omega)}$ is defined by
\begin{equation}\label{holopote}
i_{v}\omega=\sqrt{-1}\overline{\partial}\theta_{v}^{(\omega)}
\quad\text{and}\quad  \int_{M}\theta_{v}^{(\omega)}\omega^{n}=0.
\end{equation}

\begin{lemma}\label{cv}
For any holomorphic vector field $v$ on $M$,
there exists a unique constant $c_{v}$ such that
$$
\Delta_{\omega}\pote+\pote+v(\rho_{\omega})=c_{v}
$$
where $\Delta_{\omega}$ is the negative Laplacian for $\omega$, and
$c_{v}$ is given by
$$
c_{v}=\frac{\int_{M}\theta_{v}^{(\omega)}e^{\rho_{\omega}}\omega^{n}}
{\int_{M}\omega^{n}}.
$$
\end{lemma}

\begin{proof}
Note $L_{v}\omega=\deldel\pote$ which comes from the first equation in
\eqref{holopote} and Cartan's formula for the Lie derivative.
This gives $L_{v}(\omega^{n})=(\Delta_{\omega}\pote)\omega^{n}$.
Then the derivative of the first equation in \eqref{Ricci} with respect
to $v$ yields
$$
\Delta_{\omega}\pote+\pote+v(\rho_{\omega})=c_{v}
$$ 
for some constant $c_{v}$.
Moreover, by Stokes' theorem, we have
\begin{equation}\label{intLap2}
0=
\int_{M}L_{v}(e^{\rho_{\omega}}\omega^{n})
=\int_{M}\left(\Delta_{\omega}\pote+v(\rho_{\omega})\right)
e^{\rho_{\omega}}\omega^{n}.
\end{equation}
Thus
$c_{v}=\int_{M}\pote{e^{\rho_{\omega}}}\omega^{n}/\int_{M}\omega^{n}$.
\end{proof}

\begin{lemma}\label{Fut2}
For any holomorphic vector field $v$ on $M$, we have
$$
\mathrm{Fut}_{V}^{\sigma}(v)=-\int_{M}
(\pote-c_{v})e^{-\sigma(\Pote)}\omega^{n} 
$$
where $c_{v}$ is the constant in Lemma~\ref{cv}.
\end{lemma}

\begin{proof}
By Lemma~\ref{cv}, it suffice to show 
\begin{equation}\label{intLap}
\int_{M}\left(-\Delta_{\omega}\pote+v(\sigma(\Pote))\right)
e^{-\sigma(\Pote)}\omega^{n}=0.
\end{equation}
By Stokes' theorem, we have
$\int_{M}L_v(e^{-\sigma(\Pote)}\omega^{n})=0
$
which yields \eqref{intLap}.
\end{proof}

\subsection{Multiplier Hermitian-Einstein metrics on KSM-manifolds}

We use same notation as in the Section~\ref{KSM-mfd}.
Let $\mathfrak{M}=(W;L_{1},L_{2},\dotsc,L_{l};P)$ be an
$(n,l)$-dimensional KSM-data
and $Z_{\mathfrak{M}}$ be the associated KSM-manifold with a
fiber-directed holomorphic vector field $V$.
Let $\omega_{\refe}$ be the reference K\"{a}hler metric on $\Z$ defined
in \eqref{ref}.

An $(S^{1})^{l}(\subset(\mathbb{C}^{*})^{l})$-invariant function $u$ on
$X_{P}$ can be seen as a function of
$\bm{y}=(y_{1},\dotsc,y_{l})\in\mathbb{R}^{l}$ on
$(\mathbb{C}^{*})^{l}=\{(t_{1},\dotsc,t_{l})\}\subset{X_{P}}$, where
$y_{i}=-\log|t_{i}|^{2}$.
Moreover, by composing with $\bm{x}=(x_{1},\dotsc,x_{l})$ defined in
\eqref{xq}, $u(\bm{x})$ can be seen as a function on
$Q_{\mathfrak{M}}\subset{Z_{\mathfrak{M}}}$. 
We denote by $\mathcal{F}_{\mathfrak{M}}$ the set of such functions on
$Q_{\mathfrak{M}}$.
For instance,
the Ricci potential $\rho_{\omega_{\refe}}$ for $\omega_{\refe}$ is
reduced to an element in $\mathcal{F}_{\mathfrak{M}}$.
Since
\begin{align*}
\omega_{\refe}^{n+l}=
\frac{(n+l)!}{n!}
&\det\left(
\frac{\partial^{2}u_{P}}{\partial{y_{i}}\partial{y_{j}}}(\bm{x})
\right)
\prod_{\alpha=1}^{n}\left(
1+\langle\bm{\mu}_{\alpha},\bm{m}_{0}(\bm{x})\rangle\right)\\
&\left(\bigwedge_{i=1}^{l}\left(\sqrt{-1}\frac{d\tau_{i}}{\tau_{i}}
\wedge\frac{d\overline{\tau_{i}}}{\overline{\tau_{i}}}\right)\right)
\wedge\left(\bigwedge_{\alpha=1}^{n}\left(\sqrt{-1}
dz_{\alpha}\wedge{d\overline{z_{\alpha}}}\right)\right)
\end{align*}
by \eqref{ref2}, 
then by definition of the Ricci potential, there exists a constant
$c\in\mathbb{R}$ such that
\begin{equation}\label{HessuP}
e^{-\rho_{\omega}-c}=
\frac{\omega_{\refe}^{n+l}}{\eta_{\refe}}=
e^{u_{P}(\bm{x})}
\det\left(
\frac{\partial^{2}u_{P}}{\partial y_{i}\partial y_{j}}(\bm{x})\right)
\prod_{\alpha=1}^{n}
\left(1+\langle\bm{\mu}_{\alpha},\bm{m}_{0}(\bm{x})\rangle\right)
\end{equation}
on 
$\pi_{\mathfrak{M}}^{-1}(w)(\subset{Q_{\mathfrak{M}}})$
where
$\bm{m}_{0}(\bm{x})=\left(
\frac{\partial{u_{P}}}{\partial{y_1}}(\bm{x}),\dotsc,
\frac{\partial{u_{P}}}{\partial{y_{l}}}(\bm{x})\right)$.

The complex torus $(\mathbb{C}^{*})^{l}$ naturally acts on $\Z$ and
generates fiber-directed holomorphic vector fields.
Let $v_{i}$ be the fiber-directed holomorphic vector field on $\Z$ which
corresponds to
$t_{i}\frac{\partial}{\partial
t_{i}}\in\mathrm{Lie}((\mathbb{C}^{*})^{l})$.
Then, by definition, its potential function
$\theta_{v_{i}}^{(\omega_{\refe})}$ with respect to $\omega_{\refe}$ is
reduced to an element in $\mathcal{F}_{\mathfrak{M}}$, and in fact
\begin{equation}\label{bv}
\theta_{v_{i}}^{(\omega_{\refe})}(\bm{x})
=-\frac{\partial{u_{P}}}{\partial{y_{i}}}(\bm{x})+b_{v_{i}}
\end{equation}
on $Q_{\mathfrak{M}}$ for some constant $b_{v_i}\in\mathbb{R}$.

\begin{lemma}\label{thetacv}
We have 
$$
\theta_{v_{i}}^{(\omega_{\refe})}(\bm{x})
=-\frac{\partial u_{P}}{\partial y_{i}}(\bm{x})+c_{v_{i}}
$$
on $Q_{\mathfrak{M}}$, where $c_{v_{i}}$ is the unique constant in
Lemma~\ref{cv}.
\end{lemma}

\begin{proof}
We compare $b_{v_i}$ in \eqref{bv} and $c_{v_i}$.
It follows from the equation \eqref{intLap2} in Lemma~\ref{cv} that
$$
\int_{\Z}( -\theta_{v_{i}}^{(\omega_{\refe})}+c_{v_{i}})
e^{\rho_{\omega_{\refe}}}\omega_{\refe}^{n+l}=0.
$$
Recall  each fiber of $\widetilde{\pi}_{\mathfrak{M}}:
Z_{\mathfrak{M}}\to W$ is isomorphic to the toric Fano manifold $X_{P}$.
Note that in the above integral, the integrals along each fiber are all
equal. Thus by \eqref{HessuP}, the above equality is reduced to
\begin{eqnarray*}
0
&=&\int_{\mathbb{R}^{l}} ( -\theta_{v_{i}}^{(\omega_{\refe})}(\bm{y})
+c_{v_{i}})e^{-u_{P}(\bm{y})}\bm{dy} \\
&=&\int_{\mathbb{R}^{l}} \left(\frac{\partial u_{P}}{\partial y_{i}}
(\bm{y})-b_{v_i}+c_{v_i}\right) e^{-u_P(\bm{y})}\bm{dy} \\
&=&\int_{\mathbb{R}^{l}} (-b_{v_i}+c_{v_i}) e^{-u_P(\bm{y})}\bm{dy}.
\end{eqnarray*}
Therefore $b_{v_i}=c_{v_i}.$
\end{proof}

According to the above lemma, 
the potential function $\theta_{V}^{(\omega_{\refe})}$ for the
fiber-directed holomorphic vector field $V$ on $\Z$
is written as
\begin{equation}\label{thetaV}
\theta_{V}^{(\omega_{\refe})}(\bm{x})
=-\sum_{k=1}^{l}c_{k}\frac{\partial u_{P}}{\partial y_{k}}(\bm{x})+C_{V}
=-\langle \bm{c},\bm{m}_{0}(\bm{x})\rangle +C_{V}
\end{equation}
on $Q_{\mathfrak{M}}$,
where $\bm{c}=(c_{1},\dotsc,c_{l})\in\mathbb{R}^{l}$ is the unique
constants satisfying $V=\sum_{k}c_{k}v_{k}$, and
$C_{V}=\sum_{k}c_{k}c_{v_{k}}$

\begin{theorem}\label{ThmA2}
For a fiber-directed holomorphic vector field $V$ on $Z_{\mathfrak{M}}$,
the followings are equivalent.
\begin{enumerate}
\item[(i)] The KSM-manifold $Z_{\mathfrak{M}}$ admits an
	   $(S^{1})^{l}$-invariant multiplier Hermitian-Einstein metric 
$\widetilde{\omega}=
\omega\exp(-\frac{1}{n}\sigma(\theta_{V}^{(\omega)}))$
	   of type $(\sigma,V)$.
\item[(ii)] The $(\sigma,V)$-Futaki invariant vanishes for any
	    fiber-directed holomorphic vector field.
\item[(iii)] For any $k=1,2,\dotsc, l$, we have
$$
\int_{P^{*}}z_k\prod_{\alpha=1}^{n}
(1+\langle \bm{\mu}_{\alpha}, \bm{z} \rangle ) 
e^{-\sigma (-\langle \bm{c},\bm{z} \rangle +C_{V})}\bm{dz}=0.
$$
\end{enumerate}
\end{theorem}

\begin{proof}
It suffices to show the two directions (ii) $\Rightarrow$ (iii) and
(iii) $\Rightarrow$ (i).
First, we prove (ii) $\Rightarrow$ (iii).
By Lemma~\ref{Fut2}, for the fiber-directed holomorphic vector field
$v_{i}$ on $\Z$ corresponds to
$t_{i}\frac{\partial}{\partial t_{i}}\in\mathrm{Lie}
((\mathbb{C}^{*})^{l})$,
we have
$$
\mathrm{Fut}_{\Z}^{\sigma}(v_{i})
=\int_{\Z}( \theta_{v_{i}}^{(\omega_{\refe})}-c_{v_{i}} )
e^{-\sigma(\Poteo)} \omega_{\refe}^{n+l}=0.
$$
Since the integral along a fiber equals to one another, it follows from
Lemma~\ref{thetacv} that the above equality is reduced to
$$
\int_{\mathbb{R}^{l}}\frac{\partial u_{P}}{\partial y_{i}}
e^{-\sigma(-\langle \bm{c},\nabla u_{P} \rangle +C_{V})}
\det(\nabla^{2}u_{P}) 
\prod_{\alpha=1}^{n}(1+\langle\bm{\mu}_{\alpha},
\nabla u_{P}\rangle)\bm{dy}=0.
$$
Recall that the image of
$\bm{m}_0(\bm{x})=\left(\frac{\partial{u_{P}}}{\partial{y_{1}}}(\bm{x}),
\dotsc,\frac{\partial{u_{P}}}{\partial{y_{l}}}(\bm{x})\right)$
on a fiber $\widetilde{\pi}^{-1}_{\mathfrak{M}}(w)$ of $\Z$ is the dual
polytope $P^{*}$.
By taking the Legendre transformation
$z_{i}=\frac{\partial u_{P}}{\partial y_{i}}(\bm{y})\,\,(i=1,\dotsc,l)$,
we have
$$
\int_{P^{*}}z_k\prod_{\alpha=1}^{n}
(1+\langle \bm{\mu}_{\alpha}, \bm{z} \rangle ) 
e^{-\sigma (-\langle \bm{c},\bm{z} \rangle +C_{V})}\bm{dz}=0.
$$

Secondly we prove (iii) $\Rightarrow$ (i).
As in the proof of Theorem~\ref{KSMFano}, take any
$\varphi\in\mathcal{F}_{\mathfrak{M}}$,
and put $u=u_P+\varphi$ and
$$\eta_{u}=
\frac{(n+1)!}{n!}e^{-u(\bm{x})}
\bigwedge_{i=1}^{l}\left(\sqrt{-1}\frac{d\tau_{i}}{\tau_{i}}\wedge
\frac{d\overline{\tau_{i}}}{\overline{\tau_{i}}}\right) 
\wedge (\pi^{*}_{\mathfrak{M}}\nu_{0})^{n}.
$$
Then $\eta_{u}$ defines a volume form on
$Q_{\mathfrak{M}}$ and extends that on $\Z$.
If $u(\bm{y})$ is strictly convex on $\mathbb{R}^{l}$, 
then $\omega_{u}:=-\deldel\log\eta_{u}\in{2\pi{c_1(\Z)}}$ defines a
K\"{a}hler metric on $Q_{\mathfrak{M}}$.
Note that
$$
\theta_{V}^{(\omega_u)}(\bm{x}) 
=\theta_{V}^{(\omega_{\refe})}(\bm{x}) + V(\varphi(\bm{x}))
=-\sum_{k=1}^{l}c_{k}\frac{\partial u}{\partial y_{k}}(\bm{x})+C_{V}
$$
on $Q_{\mathfrak{M}}$.
If $u=u_{P}+\varphi\,(\varphi\in\mathcal{F}_{\mathfrak{M}})$ solves the
equation of a multiplier Hermitian-Einstein metric,
then by a same calculation as to get the equation \eqref{HessuP}, we
have
$$
\det\left(
\frac{\partial^{2}u}{\partial y_{i}\partial y_{j}}(\bm{x})\right)
\prod_{\alpha=1}^{n}\left(  1+\sum_{k=1}^{l} \mu_{\alpha}^{(k)}
\frac{\partial u}{\partial y_{k}}(\bm{x}) \right)
=e^{-u(\bm{x})+\sigma(\theta_{V}^{(\omega_{u})}(\bm{x}))}.
$$
Therefore, in order to construct a multiplier Hermitian-Einstein metric
on $\Z$, it suffice to solve the real Monge-Amp\`{e}re equation
\begin{equation}\label{gMAeq}
g\left(\nabla u(\bm{y})\big)\det\big(\nabla^{2}u(\bm{y})\right)
=e^{-u(\bm{y})}
\end{equation}
for a convex function $u$ on $\mathbb{R}^{l}$, where
\begin{equation}\label{g-function}
g(\bm{z}):=\prod_{\alpha=1}^{n}
\left(1+\langle \bm{\mu}_{\alpha},\bm{z}\rangle\right)
e^{-\sigma(-\langle \bm{c},\bm{z}\rangle +C_{V})}
\end{equation}
with $\bm{z}=\nabla u(\bm{y})$.
By a result of Berman-Berndtsson \cite[Theorem~1.1]{BB13},
the solvability of \eqref{gMAeq} is guaranteed
under the assumption that the barycenter of $P^{*}$ with respect to 
the measure $g(\bm{z})\bm{dz}$ is equal to the origin.
This assumption is nothing but the condition (iii).
Moreover the same argument as in \cite[Section~3.8]{BB13} allows us to
conclude that 
$u$ defines a smooth multiplier Hermitian-Einstein metric on $\Z$ of
type $(\sigma,V)$
(see also \cite[Theorem~1.1]{SaTa} for another regularity argument by an
application of a geometric flow, which is based on
\cite{SoTi17,SzTo11}).
This completes the proof of Theorem~\ref{ThmA2}.
\end{proof}

\section{The $(\sigma,V)$-Ding functional and the
$(\sigma,V)$-D-polystability for KSM-manifolds}\label{section:stability}

In this section, we introduce the notion of the $(\sigma,V)$-D-stability
for KSM-manifolds to prove Theorem~\ref{ThmB}.
Let $\mathfrak{M}=(W;L_{1},L_{2},\dotsc,L_{l};P)$ be an
$(n,l)$-dimensional KSM-data
and $Z_{\mathfrak{M}}$ be the associated KSM-manifold with a
fiber-directed holomorphic vector field $V$.
We use same notations as in the previous sections.

\subsection{The $(\sigma,V)$-Ding functional for $\Z$}

As in the proof of Theorem~\ref{ThmA2},
the equation to construct a multiplier Hermitian-Einstein metric on a
KSM-manifold $\Z$ is reduced to that for convex functions on
$\mathbb{R}^{l}$.
The $(\sigma,V)$-Ding functional $\mathrm{Ding}_{V}^{\sigma}$ should
also be reduced to a functional for convex functions.
Following \cite{BB13, CGSZ19} (see also \cite[Section~3.4]{Ya21}),
we define classes of convex functions on $\mathbb{R}^{l}$ to discuss
$\mathrm{Ding}_{V}^{\sigma}$ on $\Z$.
Let $v_{P^{*}}(\bm{y})=\sup_{\bm{z}\in P^{*}}\langle \bm{y},\bm{z}
\rangle$ be the support function for the dual polytope $P^{*}$ of an
$l$-dimensional Fano polytope $P$.
For a convex function $u(\bm{y})$ on $\mathbb{R}^{l}$, let
$u^{*}(\bm{z}):=\sup_{\bm{y}\in\mathbb{R}^{l}}
(\langle \bm{y},\bm{z}\rangle - u(\bm{y}))$
be its Legendre dual. We set
\begin{align*}
&\mathrm{PSH}(P^{*}):=\Set{u:\mathbb{R}^{l}\to\mathbb{R}\,;\,
\text{convex}\,|\,u\leqq v_{P^{*}}+C\text{ on }
\mathbb{R}^{l} \text{ for some constant }C},\\
&\mathcal{E}^{1}(P^{*}):=\Set{u\in\mathrm{PSH}(P^{*})\,|\,
\int_{P^{*}}u^{*}\bm{dz} < + \infty },\\
&\mathrm{PSH}_{b}(P^{*}):=\Set{u : \mathbb{R}^{l}\to\mathbb{R}\,;\,
\text{convex}\,|\,|u-v_{P^{*}}|<C \text{ on } \mathbb{R}^{l}
\text{ for some constant } C}.
\end{align*}
An element $u\in\mathrm{PSH}(P^{*})$ corresponds to a torus invariant
plurisubharmonic (psh) metric $e^{-u}$
on the anti-canonical bundle of the toric Fano manifold $X_{P}$
associated with $P$.
An element in $\mathcal{E}^{1}(P^{*})$ is characterized by the
finiteness of the Monge-Amp\'{e}re energy for $X_{P}$.
An element in $\bPSH$ corresponds to a bounded psh function on $X_{P}$.
In these terminology, we have
$\bPSH\subset\Eone\subset\mathrm{PSH}(P^{*})$.
Recall that $u\in\mathrm{PSH}(P^{*})$ defines a psh metric for the
KSM-manifold $\Z$ as in the proof of Theorems~\ref{KSMFano} and
\ref{ThmA2}.
Indeed, we define on $Q_{\mathfrak{M}}$,
$$
\eta_{u}=
\frac{(n+l)!}{n!}e^{-u(\bm{x})}
\bigwedge_{i=1}^{l}\left(\sqrt{-1}
\frac{d\tau_{i}}{\tau_{i}}\wedge
\frac{d\overline{\tau_{i}}}{\overline{\tau_{i}}}\right)
\wedge (\pi^{*}_{\mathfrak{M}}\nu_{0})^{n},
$$
and this can be extended to an invariant psh metric on the
anti-canonical bundle of $\Z$.

For $u\in\mathcal{E}^{1}(P^{*})$, we set
\begin{equation}\label{ding-KSM}
D_V^{\sigma}(u):=
\frac{1}{|P^{*}|_{V}^{\sigma}}\int_{P^{*}}u^{*}(\bm{z})g(\bm{z})\bm{dz}
-\log\int_{\mathbb{R}^{l}}e^{-u(\bm{y})}\bm{dy},
\end{equation}
where $g(\bm{z})$ is the function defined by \eqref{g-function},
$\bm{c}$ and $C_{V}$ are unique constants determined by \eqref{thetaV}
and $|P^{*}|_{V}^{\sigma}:=\int_{P^{*}}g(\bm{z})\bm{dz}$
is the volume.
It is easy to see that the functional $D_{V}^{\sigma}$ is the
restriction of the $(\sigma,V)$-Ding functional
$\mathrm{Ding}_{V}^{\sigma}$
(modulo a positive multiplicative constant and an additive constant)
to $\mathcal{E}^{1}(P^{*})$.
Indeed we have
$$
\delta D_{V}^{\sigma}(\delta u)
=\int_{\mathbb{R}^{l}}-\delta u 
\left(  
\frac{1}{|P^{*}|_{V}^{\sigma}}g(\nabla u)
\det(\nabla^{2}u)-\frac{e^{-u}}{\int_{\mathbb{R}^{l}}e^{-u}}
\right)\bm{dy}
$$
for a smooth strictly convex function $u$
by using the change of variables $\bm{z}=\nabla u(\bm{y})$ and the
derivative formula $\delta u = -\delta u^{*}(\nabla u)$,
and we see that a critical point of $D_{V}^{\sigma}$ satisfies the
(rescaled) real Monge-Amp\`{e}re equation 
to define a multiplier Hermitian-Einstein metric on $\Z$.
For simplicity, we call the functional $D_{V}^{\sigma}$ the
$(\sigma,V)$-Ding functional for $\Z$, if there is no fear of
confusion.

\subsection{Toric geodesics and the $(\sigma,V)$-D-stability}
\label{toricgeod}

Mabuchi~\cite{Ma87b} introduced the $L^{2}$ metric on the space of
smooth K\"{a}hler metrics to consider geodesics.
Smooth geodesics with respect to this $L^{2}$ metric do not
necessarily exist.
Darvas~\cite{Dar15} introduced a weak notion of geodesics, called the
finite energy geodesics.
In the case of toric Fano manifolds, the geodesics of invariant metrics
on $\mathcal{E}^{1}(P^{*})$ is considered.
These are called the {\it toric geodesics}, and are constructed by the
Legendre duality as follows.
See \cite{RWN14, SZ12} for more details.
For $u_{0}\in\Eone$, a toric geodesic ray $u_{t}$ started from $u_{0}$
is given by
$$
u_{t}=\left(u_{0}^{*}+t\phi\right)^{*}\in\Eone
$$ 
for $t\geqq{0}$, where $\phi:\mathrm{Int}(P^{*})\to\mathbb{R}$ is an
integrable convex function.
It is known that $u_{t}(\bm{y})$ is convex in $(\bm{y}, t)$.
In our case, a toric geodesic ray $u_{t}$ gives a geodesic ray in the
space of metrics on the KSM-manifold $\Z$ by considering the metric
$\eta_{u_{t}}$.

We next review the notion of test configurations.
See for instance \cite{Be16} for more details.
Let $M$ be a Fano manifold.
A test configuration $(\mathcal{M,L})$ for $(M, K_{M}^{-1})$ is a proper
normal variety $\mathcal{M}$
with a $\mathbb{Q}$-line bundle $\mathcal{L}$ and a morphism $\pi:
\mathcal{M}\to\mathbb{P}^{1}(\mathbb{C})$ satisfying the followings.
(i) There exists a linearized $\mathbb{C}^{*}$-action on
$(\mathcal{M,L})$ such that $\pi$ is equivariant with respect to the
multiplicative $\mathbb{C}^{*}$-action on $\mathbb{P}^{1}(\mathbb{C})$.
(ii) There exists an equivariant isomorphism with a trivial
$\mathbb{C}^{*}$-action between
$(\mathcal{M,L})|_{\mathbb{P}^{1}(\mathbb{C})\setminus\{0\}}$ and
$(M\times\mathbb{P}^{1}(\mathbb{C})\setminus\{0\},\pi^{*}K_{M}^{-1})$.
For a test configuration $(\mathcal{M,L})$,
a bounded psh geodesic ray started from a given bounded psh metric on
$K_{M}^{-1}$ is obtained by an envelope construction (\cite{Be16}).
The other way to construct such geodesic ray induced from
$(\mathcal{M,L})$ is known in \cite{CT08, RWN14, PS07}.

A test configuration $(\mathcal{M,L})$ for the toric Fano manifold
$M=X_{P}$ is a {\it toric test configuration}
if $\mathcal{M}$ is a toric variety itself with a
$(\mathbb{C}^{*})^{l}\times\mathbb{C}^{*}$-action,
where the first factor acts on each fiber of
$\mathcal{M}\to\mathbb{P}^{1}(\mathbb{C})$ as the action on $X_{P}$
and the second factor corresponds to the $\mathbb{C}^{*}$-action of the
test configuration.
According to Donaldson~\cite{Do02}, a toric test configuration is
obtained from a rational piecewise linear (PL) convex function on
$P^{*}$. Let
$$
\phi:=\max\{\lambda_{1},\dotsc,\lambda_{p}\}
$$
be a rational PL convex function on $P^{*}$,
where $\lambda_{r}$ is an affine function with rational coefficients and
$p$ is an integer.
Take a large integer $R$ such that
$$
\Box:=
\Set{(\bm{y}, y_{l+1})\in\mathbb{R}^{l}\oplus\mathbb{R}\,|\,
\bm{y}\in P^{*}\text{ and }0\leqq{y_{l+1}}\leqq R-\phi(\bm{y})}
$$
is an $(l+1)$-dimensional polytope,
and take a sufficiently divisible integer $r$ such that $r\Box$ is a
lattice polytope.
Then $r\Box$ defines a toric variety $\mathcal{M}$ with a line bundle
$\mathcal{L}^{\otimes{r}}$,
where $\mathcal{L}$ is a $\mathbb{Q}$-line bundle.
In fact, $(\mathcal{M,L})$ defines a toric test configuration for the
toric Fano manifold $X_{P}$.
According to Song-Zelditch~\cite{SZ12}, a toric test configuration given
by $(P,\phi,R)$ induces a toric geodesic ray in $\Eone$ defined by
\begin{equation}\label{geodfR}
u_{t}=\left(u_{0}^{*}+t(\phi-R)\right)^{*}.
\end{equation}

In the construction of a KSM-manifold $\Z$ associated to
$\mathfrak{M}=(W;L_1,\dotsc,L_{l};P)$, we can use 
$\widetilde{\mathfrak{M}}=
(W;L_1,\dotsc,L_{l},\mathcal{O}_{W};\left(r\Box\right)^{*})$,
although $\widetilde{\mathfrak{M}}$ is not a KSM-data.
Hence, we obtain
$\widetilde{\mathcal{M}}:=Z_{\widetilde{\mathfrak{M}}}
\subset\mathbb{P}(E_{\widetilde{\mathfrak{M}}})$, by abuse of notation,
and put
$$
\widetilde{\mathcal{L}}:=
\left(\mathcal{O}_{
\mathbb{P}(E_{\widetilde{\mathfrak{M}}})}(1)|_{\widetilde{\mathcal{M}}}
\right)^{\!\otimes(1/r)}\otimes
\left(\widetilde{\pi}_{\widetilde{\mathfrak{M}}}^{*}K_{W}^{-1}\right),
$$
where $\mathcal{O}_{\mathbb{P}(E_{\widetilde{\mathfrak{M}}})}(1)$ is the
relative hyperplane-section line bundle over
$\mathbb{P}(E_{\widetilde{\mathfrak{M}}})$.
Then the pair $(\widetilde{\mathcal{M}},\widetilde{\mathcal{L}})$
becomes a test configuration for $(\Z,K_{\Z}^{-1})$.
In this paper, we call such a test configuration for $(\Z,K_{\Z}^{-1})$
a {\it fiber-directed toric test configuration} for $\Z$.
A fiber-directed toric test configuration
$(\widetilde{\mathcal{M}},\widetilde{\mathcal{L}})$ for $\Z$ induces the
geodesic ray $\eta_{u_{t}}$
in the space of metrics on the anti-canonical line bundle for $\Z$,
where $u_{t}$ is the geodesic ray \eqref{geodfR}.

Now we observe the limit slope of the $(\sigma,V)$-Ding functional
$D_{V}^{\sigma}$ along a geodesic ray $\eta_{u_{t}}$ to define an
algebraic stability for $\Z$.

\begin{proposition}\label{Dslope}
Let $u_{t}:=(u_{P}^{*}+t(\phi-R))^{*}$ be the geodesic ray in $\Eone$
which is induced by a toric test configuration given by $(P,\phi,R)$.
Here $u_{P}$ is defined in \eqref{uP}.
Then we have
$$
\lim_{t\to\infty}\frac{D_{V}^{\sigma}(u_{t})}{t}=
\frac{1}{|P^{*}|_V^{\sigma}}
\int_{P^{*}}\phi(\bm{z})g(\bm{z})\bm{dz}-\phi(\bm{0})
$$
where $g(\bm{z})$ is the function defined in \eqref{g-function}.
\end{proposition}

\begin{proof}
It is easy to see that
$$
\lim_{t\to\infty}\frac{1}{t}\int_{P^*}u_{t}^{*}(\bm{z})g\bm{dz}=
\int_{P^{*}}(\phi(\bm{z})-R)g(\bm{z})\bm{dz}.
$$
On the other hand, by \cite[Thoerem~14]{Ya21}, the limit slope of the
log term is given by
$$
\lim_{t\to\infty}\frac{1}{t}
\log\int_{\mathbb{R}^{l}}e^{-u_{t}(\bm{z})}\bm{dy}=\phi(\bm{0})-R.
$$
This completes the proof.
\end{proof}

For a fiber-directed toric test configuration
$(\widetilde{\mathcal{M}},\widetilde{\mathcal{L}})$ for $\Z$
given by $(P,\phi,R)$,
we define the {\it $(\sigma,V)$-Ding invariant} as
$$
\mathfrak{D}_{V}^{\sigma}
(\widetilde{\mathcal{M}},\widetilde{\mathcal{L}})
\left(=\mathfrak{D}_{V}^{\sigma}(\phi)\right):=
\frac{1}{|P^{*}|_{V}^{\sigma}}\int_{P^{*}}\phi(\bm{z})
g(\bm{z})
\bm{dz}-\phi(\bm{0}),
$$
where $g(\bm{z})$ is the function defined in \eqref{g-function}.

\begin{definition}
The KSM-manifold $\Z$ with fiber-directed $V$ is
{\it fiber-directed relative $(\sigma,V)$-D-polystable} if
$\mathfrak{D}_{V}^{\sigma}
(\widetilde{\mathcal{M}},\widetilde{\mathcal{L}})\geqq{0}$ for any
fiber-directed
toric test configuration
$(\widetilde{\mathcal{M}},\widetilde{\mathcal{L}})$,
and the equality holds if and only if the rational PL convex function
defining $(\widetilde{\mathcal{M}},\widetilde{\mathcal{L}})$ is affine
linear.
\end{definition}

Now we prove Theorem~\ref{ThmB}
which is a counter part of Theorem~\ref{ThmA} from view point of the
fiber-directed relative $(\sigma,V)$-D-polystability.

\begin{proof}[Proof of Theorem~\ref{ThmB}]
We consider the barycenter $\bm{b}_{g}$ of the dual polytope
$P^{*}$ with respect to the probability measure
$g\bm{dz}/|P^{*}|_{V}^{\sigma}$, that is,
$\bm{b}_{g}:=\frac{1}{|P^{*}|_{V}^{\sigma}}
\int_{P^{*}}\bm{z}g(\bm{z})\bm{dz}$.
Then the integral condition in Theorem~\ref{ThmA} is equivalent to
$\bm{b}_{g}=\bm{0}$.
We use Jensen's inequality which means that
$$
\phi(\bm{b}_{g})\leqq
\frac{1}{|P^{*}|_{V}^{\sigma}}\int_{P^{*}}\phi(\bm{z})g(\bm{z})\bm{dz}
$$
for any convex function $\phi$ on $P^{*}$, and that the equality holds
if and only if $\phi$ is affine linear.
Thus the condition $\bm{b}_{g}=\bm{0}$ implies the
fiber-directed relative ($\sigma, V$)-D-polystability for $\Z$.

Conversely, considering fiber-directed toric test configurations given
by $(P,\phi,R)$ and $(P,-\phi,R)$ with the coordinate function
$\phi(\bm{z})=z_{k}$,
we have the integral condition in Theorem~\ref{ThmA}.
This completes the proof.
\end{proof}

\section{Coercivity for the $(\sigma,V)$-Ding functional}
\label{section:coercive}

The coercivity estimate in Theorem~\ref{ThmC} plays an 
crucial role in the result of
Berman-Berndtsson~\cite[Theorem~1.1]{BB13} used in the proof of
Theorem~\ref{ThmA}. Theorem~\ref{ThmC} can be considered as an energy
theoretic version of Theorem~\ref{ThmA}.
Indeed, we emphasize that we do not use multiplier Hermitian-Einstein
metrics to prove Theorem~\ref{ThmC}.

\subsection{Reduced $J$-functional and $(\sigma,V)$-$J$-functional}

First, we define a function $h(\bm{z})$ on $P^{*}$ by
\begin{equation}\label{h-function}
h(\bm{z}):=\prod_{\alpha=1}^{n}\left(1+\langle
\bm{\mu}_{\alpha},\bm{z}\rangle\right).
\end{equation}
Then we have
$g(\bm{z})=
h(\bm{z})\exp\left(-\sigma(-\langle\bm{c},\bm{z}\rangle+C_{V})\right)$,
where $g(\bm{z})$ is the function defined by \eqref{g-function}.
Note that both $h(\bm{z})$ and $g(\bm{z})$ are positive functions on
$P^{*}$.
In order to define the coercivity for $D_{V}^{\sigma}$,
we introduce the {\it reduced $J$-functional}
$J_{\mathrm{red}}$ and
the {\it reduced $(\sigma,V)$-$J$-functional}
$J_{V,\mathrm{red}}^{\sigma}$ on $\Eone$ as follows:
\begin{align*}
&\JfctORG(u):=\inf_{\ell}
\left\{\frac{1}{|P^{*}|_{h}}
\int_{P^{*}}(u^{*}(\bm{z})-\ell(\bm{z}))h(\bm{z})\bm{dz}
-\inf_{P^{*}}(u^{*}-\ell)\right\},\\
&\Jfct(u):=\inf_{\ell}
\left\{\frac{1}{|P^{*}|_{V}^{\sigma}}
\int_{P^{*}}(u^{*}(\bm{z})-\ell(\bm{z}))g(\bm{z})\bm{dz}
-\inf_{P^{*}}(u^{*}-\ell)\right\},
\end{align*}
where $\ell$ runs through arbitrary affine functions on $P^{*}$,
and $|P^{*}|_{h}:=\int_{P^{*}}h(\bm{z})\bm{dz}$.

\begin{definition}\label{def:coercive}
The $(\sigma,V)$-Ding functional $D_{V}^{\sigma}$ is {\it coercive} if
there exist positive constants $\delta,C>0$ such that 
$$
D_{V}^{\sigma}(u)\geqq\delta\JfctORG(u)-C
$$
for any $u\in\Eone$.
\end{definition}

\noindent
For a convex function $\psi$ on a convex subset $E$ of $\mathbb{R}^{l}$
and $z_{0}\in{E}$, we put
$$
\partial\psi(\bm{z}_{0}):=
\left\{\bm{a}\in\mathbb{R}^{l}\,|\,
\psi(\bm{z})\geqq\langle\bm{a},\bm{z}-\bm{z}_{0}\rangle+u(\bm{z}_{0})
\text{ on }E\right\}.
$$
The convexity of $\psi$ implies
$\partial\psi(\bm{z}_{0})\ne\varnothing$
for any $\bm{z}_{0}\in{E}$. Since
$$
\exp\left(-\sigma(-\langle\bm{c},\bm{z}\rangle+C_{V})\right)>0
$$
on $P^{*}$, there exist positive constants $A,B>0$ such that
\begin{equation}\label{positivity}
A\leqq
\exp\left(-\sigma(-\langle\bm{c},\bm{z}\rangle+C_{V})\right)
\leqq{B}\qquad
(\bm{z}\in{P^{*}}).
\end{equation}
In view of an argument similar to that
in \cite[Proof of Propositon~27]{Ya21},
this estimate \eqref{positivity} allows us to obtain
\begin{equation}\label{J-relation1}
\frac{|P^{*}|_{h}A}{|P^{*}|_{V}^{\sigma}}\JfctORG(u)\leqq
\Jfct(u)\leqq\frac{|P^{*}|_{h}B}{|P^{*}|_{V}^{\sigma}}\JfctORG(u),
\end{equation}
for any $u\in\Eone$.

Now, we can prove the following lemma:

\begin{lemma}\label{J-relation2}
If the integral condition \eqref{barycenter} holds, then we have
$$
\JfctORG(u)\leqq
\frac{1}{|P^{*}|_{h}}
\int_{P^{*}}\left(u^{*}(\bm{z})
-\langle\bm{a},\bm{z}\rangle-u^{*}(\bm{0})\right)h(\bm{z})\bm{dz}
\leqq\frac{B}{A}\JfctORG(u),
$$
for any $u\in\Eone$ and $\bm{a}\in\partial{u^{*}}(\bm{0})$.
\end{lemma}

\begin{proof}
For any $u\in\Eone$ and $\bm{a}\in\partial{u^{*}}(\bm{0})$,
the same argument as that in \cite[Proposition~27]{Ya21}
allows us to obtain
\begin{align*}
&\JfctORG(u)=
\frac{1}{|P^{*}|_{h}}
\int_{P^{*}}u^{*}(\bm{z})h(\bm{z})\bm{dz}-u^{*}(\bm{b}_{h}),\\
&\Jfct(u)=
\frac{1}{|P^{*}|_{V}^{\sigma}}
\int_{P^{*}}u^{*}(\bm{z})g(\bm{z})\bm{dz}-u^{*}(\bm{b}_{g}),
\end{align*}
where
$\bm{b}_{h}:=\frac{1}{|P^{*}|_{h}}\int_{P^{*}}\bm{z}h(\bm{z})\bm{dz}$
and
$\bm{b}_{g}:=
\frac{1}{|P^{*}|_{V}^{\sigma}}\int_{P^{*}}\bm{z}g(\bm{z})\bm{dz}$
are the barycenters of $P^{*}$ with respect to the probability measures
$h\bm{dz}/|P^{*}|_{h}$ and $g\bm{dz}/|P^{*}|_{V}^{\sigma}$,
respectively. Note that the condition \eqref{barycenter} is equivalent
to $\bm{b}_{g}=\bm{0}$. Since
$u^{*}(\bm{z})\geqq\langle\bm{a},\bm{z}\rangle+u^{*}(\bm{0})$
on $P^{*}$, we have
\begin{align*}
u^{*}(\bm{b}_{h})&\geqq
\langle\bm{a},\bm{b}_{h}\rangle+u^{*}(\bm{0})\\
&=
\frac{1}{|P^{*}|_{h}}\int_{P^{*}}
\left(\langle\bm{a},\bm{z}\rangle+u^{*}(\bm{0})\right)h(\bm{z})\bm{dz}.
\end{align*}
Therefore, we have
\begin{align*}
\JfctORG(u)&=
\frac{1}{|P^{*}|_{h}}
\int_{P^{*}}u^{*}(\bm{z})h(\bm{z})\bm{dz}-u^{*}(\bm{b}_{h}),\\
&\leqq
\frac{1}{|P^{*}|_{h}}\int_{P^{*}}u^{*}(\bm{z})h(\bm{z})\bm{dz}-
\frac{1}{|P^{*}|_{h}}
\int_{P^{*}}
\left(\langle\bm{a},\bm{z}\rangle+u^{*}(\bm{0})\right)h(\bm{z})\bm{dz}\\
&=\frac{1}{|P^{*}|_{h}}
\int_{P^{*}}\left(u^{*}(\bm{z})
-\langle\bm{a},\bm{z}\rangle-u^{*}(\bm{0})\right)h(\bm{z})\bm{dz}.
\end{align*}

On the other hand, in view of the condition \eqref{barycenter} and the
inequalities \eqref{positivity} and \eqref{J-relation1}, we obtain
\begin{align*}
&\frac{1}{|P^{*}|_{h}}
\int_{P^{*}}\left(u^{*}(\bm{z})
-\langle\bm{a},\bm{z}\rangle-u^{*}(\bm{0})\right)h(\bm{z})\bm{dz}\\
&\leqq\frac{1}{|P^{*}|_{h}A}
\int_{P^{*}}\left(u^{*}(\bm{z})
-\langle\bm{a},\bm{z}\rangle-u^{*}(\bm{0})\right)g(\bm{z})\bm{dz}\\
&=\frac{1}{|P^{*}|_{h}A}
\int_{P^{*}}\left(u^{*}(\bm{z})
-u^{*}(\bm{b}_{g})\right)g(\bm{z})\bm{dz}\\
&=\frac{|P^{*}|_{V}^{\sigma}}{|P^{*}|_{h}A}\Jfct(u)\\
&\leqq\frac{B}{A}\JfctORG(u).
\end{align*}
This completes the proof.
\end{proof}

Finally we prove Theorem~\ref{ThmC} to conclude this section.

\begin{proof}[Proof of Theorem~\ref{ThmC}]
First, we assume that $\int_{P^{*}}z_{k}g\bm{dz}=0$ for any $k$, that
is, $\bm{b}_{g}=\bm{0}$, where $\bm{b}_{g}$ is the barycenter of $P^{*}$
with respect to the measure $\frac{1}{\Pvol}g\bm{dz}$.
For an arbitrary $u\in\Eone$, since $\bm{b}_{g}=\bm{0}$,
a simple calculation allows us to obtain
$$
D_{V}^{\sigma}((u^{*}+\ell)^{*})=D_{V}^{\sigma}(u),
$$
where $\ell$ is an any affine function.
Fix $\bm{a}\in\partial{u^{*}}(\bm{0})$ and we put
$$
\widetilde{u}^{*}(\bm{z}):=
u^{*}(\bm{z})-
\langle\bm{a},\bm{z}\rangle-u^{*}(\bm{0}).
$$
Then we have
$\widetilde{u}^{*}(\bm{z})\geqq\widetilde{u}^{*}(\bm{0})=0$ and
$D_{V}^{\sigma}(u)=D_{V}^{\sigma}((\widetilde{u}^{*})^{*})$.
In view of \eqref{positivity}, an argument similar to that in
\cite[Proposition~4.2]{Na19} allows us to obtain
$$
D_{V}^{\sigma}((\widetilde{u}^{*})^{*})\geqq
\delta\frac{1}{|P^{*}|_{h}}\int_{P^{*}}
\widetilde{u}^{*}(\bm{z})h(\bm{z})\bm{dz}-C
$$
for some positive constants $\delta,C>0$.
In view of Lemma~\ref{J-relation2}, we have
\begin{align*}
\frac{1}{|P^{*}|_{h}}\int_{P^{*}}\widetilde{u}^{*}(\bm{z})\bm{dz}&=
\frac{1}{|P^{*}|_{h}}\int_{P^{*}}
\left(u^{*}(\bm{z})-
\langle\bm{a},\bm{z}\rangle-u^{*}(\bm{0})\right)h(\bm{z})
\bm{dz}\geqq\JfctORG(u).
\end{align*}
Hence, we have the coercivity for $D_{V}^{\sigma}$.

Conversely, we assume that for some $k$ ($1\leqq{k}\leqq{l}$),
$$
a_{k}:=\frac{1}{|P^{*}|_{V}^{\sigma}}\int_{P^{*}}z_{k}g(\bm{z})\bm{dz}
\ne{0}.
$$
For an affine function $\ell$, by the simple calculation, we have
$$
D_{V}^{\sigma}(\ell^{*})=-\log{N_{0}}+
\frac{1}{|P^{*}|_{V}^{\sigma}}\int_{P^{*}}\ell(\bm{z})g(\bm{z})\bm{dz},
$$
where $N_{0}$ is the number of vertices of $P^{*}$
We put $\ell(\bm{z}):=-ra_{k}z_{k}$ for a positive constant $r>0$.
If we consider the case $r\to+\infty$, then we have
$$
D_{V}^{\sigma}(\ell^{*})=-\log{N_{0}}-ra_{k}^{2}\to-\infty,
$$
while $\JfctORG(u)\geqq{0}$ for any $u\in\Eone$.
This completes the proof of Theorem~\ref{ThmC}.

%
\end{proof}


A KSM-manifold $\Z$ is said to be
{\it fiber-directed uniformly relative D-stable} with respect to
$(\sigma,V)$, where $V$ is a fiber-directed holomorphic vector field on
$\Z$, if there exist a positive constant $\lambda>0$ such that 
$$
\mathfrak{D}_{V}^{\sigma}(\phi)\geqq\lambda\JfctORG(\phi^{*}),
$$
for any rational PL convex function $\phi$ on $P^{*}$.
Here $\mathfrak{D}_{V}^{\sigma}$ is the $(\sigma,V)$-Ding invariant
defined in Section~\ref{section:stability}.

Since the function $g(\bm{z})$ satisfies $g(\bm{z})\geqq{C}$ on
$P^{*}$ for some positive constant $C>0$,
Lemma~\ref{J-relation2} also allows us to obtain the following theorem:

\begin{theorem}
Suppose a holomorphic vector field $V$ on a KSM-manifold $\Z$ is
fiber-directed.
Then $\Z$ is fiber-directed uniformly relative D-stable
with respect to $(\sigma,V)$
if and only if the integral condition \eqref{barycenter} holds.
\end{theorem}

\begin{remark}
The fiber-directed uniformly relative D-stability also follows from the
coercivity estimate of $D_{V}^{\sigma}$.
Indeed we can consider the geodesic ray \eqref{geodfR} to calculate the
limit slope.
\end{remark}

\section{Non-uniformly stable case}\label{NUStable}

In this section, we consider the non-uniformly stable case,
inspired by \cite[Sections~6 and 7]{Ya21}.
Namely, let $\sigma(s)$ be a real-valued smooth function on the interval
$I=(\alpha,\beta)$ satisfying
\begin{align*}
&-\infty<\alpha=\min_{M}\theta_{V}^{(\omega)}
\leqq\max_{M}\theta_{V}^{(\omega)}<\beta\leqq+\infty\\
&\dot{\sigma}\leqq{0}\leqq\ddot{\sigma},\\
&\lim_{t\to\alpha+0}\sigma(t)=+\infty.
\end{align*}
In this case we have
$\exp\left(-\sigma(\theta_{V}^{(\omega)})\right)=0$
at some point of $M$.
Furthermore, for a KSM-manifold $\Z$ with a fiber-directed holomorphic
vector field $V$,
if the integral condition \eqref{barycenter} holds,
then we also have
$$
\mathfrak{D}_{V}^{\sigma}(\phi)\geqq{0},
$$
for any rational PL convex function $\phi$ on $P^{*}$,
and $\Z$ is said to be {\it non-uniformly relative D-stable} with
respect to $(\sigma,V)$.

In this situation, we can prove the following existence of a subsolution
of the equation for the multiplier Hermitian-Einstein metric:

\begin{lemma}\label{subsol}
Suppose a holomorphic vector field $V$ on $\Z$ is fiber-directed.
Moreover, we assume that
\begin{equation}\label{assumption}
f(t):=\exp\left(-\sigma(t)\right)\geqq
A_{0}(t-\alpha)\qquad
(\alpha\leqq{t}<\beta),
\end{equation}
for some positive constant $A_{0}>0$.
Then, there exists a smooth and strictly convex subsolution $u\in\bPSH$
of the equation for the multiplier Hermitian-Einstein metric
in the sense that
$$
Cg(\nabla u)\det\left(\nabla^{2}u\right)\geqq{e^{-u}}
$$
for some positive constant $C>0$,
where $g(\bm{z})$ is the non-negative function on $P^{*}$ defined by
\eqref{g-function}.
\end{lemma}

\begin{proof}
We put $\widetilde{t}:=t-\alpha\,\,\,(\widetilde{t}\geqq{0})$, and
$\widetilde{f}(\widetilde{t}):=f(\widetilde{t}+\alpha)=f(t)$. Then we
have, for $\widetilde{t}\geqq{0}$,
\begin{align*}
&\widetilde{f}(\widetilde{t})\geqq\widetilde{f}(0)=f(\alpha)=0,\\
&\widetilde{f}(\widetilde{t})\geqq{A_{0}\widetilde{t}}.
\end{align*}

We put $k(\bm{z}):=-\langle\bm{c},\bm{z}\rangle+C_{V}$,
where $\bm{c}$ and ${C_{V}}$ are the unique constants defined
by \eqref{thetaV}, and
$\widetilde{k}(\bm{z}):=k(\bm{z})-\alpha$. Then we have
$g(\bm{z})=h(\bm{z})f(k(\bm{z}))$ and
$\widetilde{k}(\bm{z})\geqq{0}$ on $P^{*}$, where $h(\bm{z})$
is the positive function on $P^{*}$ defined by
\eqref{h-function}. Hence, we have
\begin{align*}
\log{f}\left(
\frac{\sum_{p\in{V(P^{*})}}k(p)
e^{\langle\bm{p},\bm{z}\rangle}}
{\sum_{p\in{V(P^{*})}}e^{\langle\bm{p},\bm{z}\rangle}}
\right)&=
\log\widetilde{f}\left(
\frac{\sum_{p\in{V(P^{*})}}\widetilde{k}(p)
e^{\langle\bm{p},\bm{z}\rangle}}
{\sum_{p\in{V(P^{*})}}e^{\langle\bm{p},\bm{z}\rangle}}
\right)\\&\geqq
\log\left(A_{0}\left(\frac{
\sum_{p\in{V(P^{*})}}\widetilde{k}(p)
e^{\langle\bm{p},\bm{z}\rangle}}
{\sum_{p\in{V(P^{*})}}e^{\langle\bm{p},\bm{z}\rangle}}
\right)\right)\\&=
\log{A_{0}}+
\log\left(\sum_{p\in{V(P^{*})}}\widetilde{k}(p)
e^{\langle\bm{p},\bm{z}\rangle}\right)-
\log\left(\sum_{p\in{V(P^{*})}}
e^{\langle\bm{p},\bm{z}\rangle}\right),
\end{align*}
where $V(P^{*})$ is the set of vertices of $P^{*}$.
In view of this inequality, the similar argument as that in
\cite[Proof of Lemma~41]{Ya21} allows us to prove Lemma~\ref{subsol}.
\end{proof}

In view of this lemma, by the same argument as that in
\cite[Proof of Theorem~41]{Ya21}
(see also \cite[Theorem~2.16]{BB13})
we can prove the following partial coercivity for
$D_{V}^{\sigma}$:

\begin{theorem}\label{MTineq}
Assume the same condition as in Lemma~\ref{subsol} and the integral
condition \eqref{barycenter}.
Then for any $\e\in (0,1)$, there exists a positive constant $C>0$ such
that 
$$
D_{V}^{\sigma}(u)\geqq
\e\frac{1}{|P|_{V}^{\sigma}}
\int_{P^{*}}u^{*}(\bm{z})g(\bm{z})\bm{dz}-C
$$
for any $u\in\Eone$ satisfying $u\geqq{u(\bm{0})}=0$.
Moreover, we also have
$$
D_{V}^{\sigma}(u)\geqq
\e\Jfct(u)-C
$$
for any $u\in\Eone$. In particular, $D_{V}^{\sigma}$ is bounded from
below.
\end{theorem}

Moreover, by the same argument as that in \cite[Theorem~1.4]{Na3} due
to the second author, we can prove the following pseudo-boundedness for
$D_{V}^{\sigma}$ without assuming the condition \eqref{assumption}:

\begin{theorem}
If the integral condition \eqref{barycenter} holds, then for an arbitrary
$\varepsilon>0$, there exists a positive constant $C_{\varepsilon}$ such
that
$$
D_{V}^{\sigma}(u)\geqq
-\varepsilon\JfctORG(u)-C_{\varepsilon},
$$
for any $u\in\Eone$.
\end{theorem}

For any $u\in\Eone$ and $K\subset\mathbb{R}^{l}$, we put
$\partial{u}(K):=\bigcup_{\bm{z}\in{K}}\partial{u}(\bm{z})$.
The {\it Monge-Amp\`{e}re measure} $\mathrm{MA}_{g}(u)$ on
$\mathbb{R}^{l}$ in the {\it sense of Alexandrov} is defined by
$$
\mathrm{MA}_{g}(u)(B):=
\frac{1}{|P^{*}|_{V}^{\sigma}}
\int_{\partial{u}(B)}g(\bm{z})\bm{dz},
$$
for a Borel set $B\subset\mathbb{R}^{l}$
(see for instance \cite[Section~3.3.1]{Ya21}).

Since
$\left\{\bm{z}\in{P^{*}}\,|\,g(\bm{z})=0\right\}=
\left\{\bm{z}\in{P^{*}}\,|\,
-\langle\bm{c},\bm{z}\rangle+C_{V}=\alpha\right\}$,
by the same argument as that in \cite[Proof of Theorem~44]{Ya21},
Theorem~\ref{MTineq} allows us
to prove the following theorem:

\begin{theorem}\label{ThmAlex}
Assume the same condition as in Lemma~\ref{subsol} and the integral
condition \eqref{barycenter}.
Then there exists $u\in\bPSH$ such that
$$
\mathrm{MA}_{g}(u)=
\frac{e^{-u}}{\int_{\mathbb{R}^{n}}e^{-u(\bm{y})}\bm{dy}}
$$
on $\mathbb{R}^{l}$, in the sense of Alexandrov.
The Legendre dual $u^{*}$ of $u$ is H\"{o}lder continuous on $P^{*}$ for
any H\"{o}lder exponent $\gamma\in(0,1)$.
Moreover, if
$\dim\left\{\bm{z}\in{P^{*}}\,|\,
-\langle\bm{c},\bm{z}\rangle+C_{V}=\alpha\right\}\leqq\frac{l}{2}$, then
$u$ is smooth and strictly convex on $\mathbb{R}^{l}$, and $\nabla{u}$
induces a diffeomorphism from $\mathbb{R}^{l}$ to
$\mathrm{Int}(P^{*})$.
\end{theorem}

%
%
%

\section{Examples}\label{section:example}

In this section, we shall give an example of a KSM-manifold which admits
non-trivial family of multiplier Hermitian-Einstein metrics, and also an
example of a non-uniformly relative D-stable KSM-manifold.

First, we put
\begin{align*}
&\sigma_{0}(t):=-t\qquad(-\infty<t<+\infty),\\
&\sigma_{1}(t):=-\log(t+1)\qquad(-1<t<+\infty),
\end{align*}
which correspond to K\"{a}hler-Ricci solitons and Mabuchi solitons,
respectively. Moreover, for any $0\leqq\tau\leqq{1}$, we put
$$
\sigma_{\tau}(t):=(1-\tau)\sigma_{0}(t)+\tau\sigma_{1}(t)
=-(1-\tau)t-\tau\log(t+1).
$$
Then $\sigma_{\tau}(t)$ is defined on $(-1,+\infty)$ for
$0\leqq\tau\leqq{1}$.

\begin{example}
Let
$Z_{1}:=\mathbb{P}\left(
\mathcal{O}_{\mathbb{P}^{1}}(1)\oplus\mathcal{O}_{\mathbb{P}^{1}}
\right)$
be the KSM-manifold associated to the $(1,1)$-dimensional KSM-data
$\mathfrak{M}=(\mathbb{P}^{1}(\mathbb{C});
\mathcal{O}_{\mathbb{P}^{1}}(1);[-1,1])$.
Now, fix any $\tau\in[0,1]$ and let $k_{\tau}^{(1)}(z)=b_{1}z+b_{2}$ be
an affine function on $[-1,1]$ corresponding to a fiber-directed
holomorphic vector $V_{\tau}^{(1)}$ on $Z_{1}$. In this case, the
normalization condition in \eqref{theta-normalization} becomes
$$
\int_{-1}^{1}\left(b_{1}z+b_{2}\right)\left(1+\frac{1}{2}z\right)dz
=\frac{1}{3}b_{1}+2b_{2}=0.
$$
Therefore, we have $b_{1}=-6b_{2}$. Furthermore, on $[-1,1]$,
$k_{\tau}^{(1)}(z)$ must
be greater than $-1$. Hence, $b_{2}$ is necessarily satisfied
$-1/7<b_{2}<1/5$. In this case, the integral condition \eqref{barycenter}
is equivalent to
$$
I_{\tau}^{(1)}(b_{2}):=
\int_{-1}^{1}z\left(1+\frac{1}{2}z\right)
\left(-6b_{2}z+b_{2}+1\right)^{\tau}
e^{(1-\tau)(-6b_{2}z+b_{2})}
=0.
$$
We want to show that we can find a solution $b_{2}\in(-1/7,1/5)$ for
the equation $I_{\tau}^{(1)}(b_{2})=0$. Simple calculation allows us to
have 
\begin{align*}
I_{\tau}^{(1)}\!\left(-\frac{1}{7}\right)&=
\left(\frac{6}{7}\right)^{\tau}e^{-\frac{1}{7}(1-\tau)}
\int_{-1}^{1}z\left(1+\frac{1}{2}z\right)\left(z+1\right)^{\tau}
e^{\frac{6}{7}(1-\tau)z}dz\\&=
\left(\frac{6}{7}\right)^{\tau}e^{-\frac{1}{7}(1-\tau)}
\int_{0}^{1}z\left\{
\left(1+\frac{1}{2}z\right)\left(1+z\right)^{\tau}
e^{\frac{6}{7}(1-\tau)z}dz\right.\\&
\qquad\qquad\qquad\qquad\qquad\quad+\left.
\left(1-\frac{1}{2}z\right)\left(1-z\right)^{\tau}
e^{-\frac{6}{7}(1-\tau)z}\right\}dz.
\end{align*}
We want to prove
$\left(1+\frac{1}{2}z\right)\left(1+z\right)^{\tau}
e^{\frac{6}{7}(1-\tau)z}
+
\left(1-\frac{1}{2}z\right)\left(1-z\right)^{\tau}
e^{-\frac{6}{7}(1-\tau)z}\geqq{0}$ for $0\leqq{z}\leqq{1}$.
This is equivalent to
$$
F_{\tau}(z):=
\left(\frac{2+z}{2-z}\right)
\left(\frac{1+z}{1-z}\right)^{\!\tau}
e^{\frac{12}{7}(1-\tau)z}
\geqq{1},
$$
for $0\leqq{z}\leqq{1}$, which clearly holds.
Hence, we conclude that $I_{\tau}^{(1)}(-1/7)>0$ for
$0\leqq\tau\leqq{1}$.

On the other hand, for $b_{2}=1/5$, simple calculation also implies that
\begin{align*}
I_{\tau}^{(1)}\!\left(\frac{1}{5}\right)&=
\left(\frac{6}{5}\right)^{\!\tau}
e^{\frac{1-\tau}{5}}
\int_{-1}^{1}z\left(1+\frac{1}{2}z\right)
\left(1-z\right)^{\!\tau}
e^{-\frac{6}{5}(1-\tau)z}dz\\&=
\left(\frac{6}{5}\right)^{\!\tau}
e^{\frac{1-\tau}{5}}
\int_{0}^{1}z\left\{
\left(1+\frac{1}{2}z\right)
\left(1-z\right)^{\!\tau}
e^{-\frac{6}{5}(1-\tau)z}\right.\\&
\qquad\qquad\qquad\qquad\quad
\left.-\left(1-\frac{1}{2}z\right)
\left(1+z\right)^{\!\tau}
e^{\frac{6}{5}(1-\tau)z}
\right\}dz.
\end{align*}
Now, we want to prove
$\left(1-\frac{1}{2}z\right)
\left(1+z\right)^{\!\tau}e^{\frac{6}{5}(1-\tau)z}-
\left(1+\frac{1}{2}z\right)\left(1-z\right)^{\!\tau}
e^{-\frac{6}{5}(1-\tau)z}\geqq{0}$, which is equivalent to
$$
G_{\tau}(z):=
\left(\frac{2-z}{2+z}\right)
\left(\frac{1+z}{1-z}\right)^{\!\tau}
e^{\frac{12}{5}(1-\tau)z}\geqq{1},
$$
for $0\leqq{z}\leqq{1}$.
We have $G_{\tau}(0)=1$ and
$$
G_{\tau}^{\prime}(z)=
\frac{2e^{\frac{12}{5}(1-\tau)z}}{5(2+z)^{2}(1-z)^{2}}
\left(\frac{1+z}{1-z}\right)^{\!\tau-1}
\left\{5\tau(2+z^{2})+2(1-\tau)(1-z^{2})(7-3z^{2})\right\}
\geqq{0},
$$
for $0\leqq{z}<1$. Hence, for $0\leqq\tau\leqq{1}$,
we get $G_{\tau}(z)\geqq{1}$, which
implies $I_{\tau}^{(1)}(1/5)<0$.
By the continuity of
$I_{\tau}^{(1)}(b_{2})$ for $-1/7\leqq{b_{2}}\leqq{1/5}$, we can find a
solution $b_{2}\in(-1/7,1/5)$ for the equation
$I_{\tau}^{(1)}(b_{2})=0$.
Therefore, by Theorem~\ref{ThmA}, the KSM-manifold
$Z_{1}=\mathbb{P}\left(
\mathcal{O}_{\mathbb{P}^{1}}(1)\oplus\mathcal{O}_{\mathbb{P}^{1}}
\right)$ admits a multiplier Hermitian-Einstein metric of type
$(\sigma_{\tau},V_{\tau}^{(1)})$ for any $0\leqq\tau\leqq{1}$.
\end{example}

\begin{example}
First of all, we note that $\sigma_{\tau}(t)$, defined as above,
satisfies the assumption \eqref{assumption} for any
$0\leqq\tau\leqq{1}$, that is, 
$$
e^{-\sigma_{\tau}(t)}=
\left(t+1\right)^{\tau}e^{(1-\tau)t}
\geqq{t+1}
\qquad
(-1\leqq{t}<+\infty).
$$
Let
$Z_{2}:=\mathbb{P}\left(
\mathcal{O}_{\mathbb{P}^{2}}(2)\oplus\mathcal{O}_{\mathbb{P}^{2}}
\right)$
be the KSM-manifold associated to the $(2,1)$-dimensional KSM-data
$\mathfrak{M}=(\mathbb{P}^{2}(\mathbb{C});
\mathcal{O}_{\mathbb{P}^2}(2);[-1,1])$.
Since $Z_{2}$ is a toric Fano manifold, $Z_{2}$ admits a
K\"{a}hler-Ricci soliton, of course. However, Mabuchi proved that
$Z_{2}$ admits no Mabuchi solitons in \cite[Example~6.6]{Ma01}.
Now, for any $\tau\in[0,1]$, let $k_{\tau}^{(2)}(z)=b_{1}z+b_{2}$ be
an affine function on $[-1,1]$ corresponding to a fiber-directed
holomorphic vector $V_{\tau}^{(2)}$ on $Z_{2}$. In this case, the
normalization condition in \eqref{theta-normalization} becomes
$$
\int_{-1}^{1}\left(b_{1}z+b_{2}\right)\left(1+\frac{2}{3}z\right)^{\!2}dz
=\frac{8}{9}b_{1}+\frac{62}{27}b_{2}=0.
$$
Therefore, we have $b_{2}=-\frac{12}{31}b_{1}$.
Furthermore, the condition $k_{\tau}^{(2)}(z)>-1$ on $[-1,1]$ is
equivalent to
$-31/19<b_{1}<31/43$. In this case, the integral condition
\eqref{barycenter} becomes
$$
I_{\tau}^{(2)}(b_{1}):=
\int_{-1}^{1}z\left(1+\frac{2}{3}z\right)^{\!2}
\left(b_{1}z-\frac{12}{31}b_{1}+1\right)^{\!\tau}
e^{(1-\tau)(b_{1}z-\frac{12}{31}b_{1})}dz=0.
$$
For $b_{1}=-31/19$, by a simple calculation, we have
\begin{align*}
&I_{0}^{(2)}\left(-\frac{31}{19}\right)=
\frac{19}{9\cdot{31^{4}}}\left(
-2268214\cdot{e^{-\frac{31}{19}}}+80048\cdot{e^{\frac{31}{19}}}
\right)<0,\\
&I_{1}^{(2)}\left(-\frac{31}{19}\right)=\frac{62}{855}>0.
\end{align*}
By the continuity of $I_{\tau}^{(2)}(-31/19)$ for $\tau\in[0,1]$, we can
find a $\tau_{0}\in(0,1)$ satisfying $I_{\tau_{0}}^{(2)}(-31/19)=0$. For
this $\tau_{0}$ and $b_{1}=-31/19$, we have
\begin{align*}
&k_{\tau_{0}}^{(2)}(z)=-\frac{31}{19}z+\frac{12}{19},\\
&k_{\tau_{0}}^{(2)}(1)=-1,\\
&k_{\tau_{0}}^{(2)}(-1)=\frac{43}{19}.
\end{align*}
Therefore, $k_{\tau_{0}}^{(2)}(z)$ satisfies the integral condition
\eqref{barycenter} and
$\exp\!\left(-\sigma_{\tau_{0}}(k_{\tau_{0}}^{(2)}(z))\right)\geqq{0}$
on $[-1,1]$. 
Hence,
$Z_{2}=\mathbb{P}\left(
\mathcal{O}_{\mathbb{P}^{2}}(2)\oplus\mathcal{O}_{\mathbb{P}^{2}}
\right)$
is non-uniformly relative D-stable with respect to
$\left(\sigma_{\tau_{0}},V_{\tau_{0}}^{(2)}\right)$.
\end{example}


\bigskip

\end{document}